\documentclass[a4paper,10pt]{article}

\usepackage[top=1in, bottom=1in, left=1in, right=1in]{geometry}
\usepackage{listings}
\usepackage{textcomp}
\usepackage{amsfonts}

\usepackage{color}
\definecolor{escol}{rgb}{0,0,0.8}
\definecolor{estcol}{rgb}{0.8,0,0}
\definecolor{afcol}{rgb}{1,0,0}

\usepackage{amsthm}
\usepackage{amsmath}
\usepackage{amssymb}
\usepackage{mathtools}
\usepackage{url}
\usepackage{enumerate}
\usepackage{array}
\usepackage{datetime}
\usepackage{authblk}

\settimeformat{ampmtime}
\newdateformat{mydate}{\THEDAY{ }\shortmonthname{ }\THEYEAR{, }\currenttime}

\usepackage{graphicx}
\usepackage{epstopdf}
\usepackage{sidecap}
\usepackage{faktor}
\usepackage{units}

\usepackage{stmaryrd}

\usepackage{bbm}

\usepackage{etoolbox}
\patchcmd{\quote}{\rightmargin}{\leftmargin 1.0cm \rightmargin}{}{}  

\usepackage[hang]{footmisc}

\usepackage{tikz}
\usetikzlibrary{arrows,matrix,positioning,fit}

\usepackage{cancel}

\newtheorem{thm}{Theorem}[section]
\newtheorem{lem}{Lemma}[section]
\newtheorem{defn}{Definition}[section]
\newtheorem{assn}{Assumption}[section]
\newtheorem{remark}{Remark}[section]
\newtheorem{coroll}{Corollary}[section]

\DeclareMathSizes{12}{12}{8}{9}

\DeclareSymbolFont{pxfontssymbolsC}{U}{pxsyc}{m}{n}
\DeclareMathSymbol{\coloneqq}{\mathrel}{pxfontssymbolsC}{66}

\numberwithin{equation}{section}

\begin{document}

\title
 {Asymptotics to all orders of the Hurwitz zeta function} 


\author[1]{Arran Fernandez \thanks{Email: \texttt{af454@cam.ac.uk}}}
\author[1,2]{Athanassios S. Fokas \thanks{Email: \texttt{T.Fokas@damtp.cam.ac.uk}}}
\affil[1]{{\small Department of Applied Mathematics \& Theoretical Physics, University of Cambridge, Cambridge, UK, CB3 0WA}}
\affil[2]{{\small Viterbi School of Engineering, University of Southern California, Los Angeles, California, 90089 USA}}




\date{}





\maketitle

\vspace{-1cm}

\begin{abstract}\noindent We present several formulae for the large-$t$ asymptotics of the modified Hurwitz zeta function $\zeta_1(x,s),x>0,s=\sigma+it,0<\sigma\leq1,t>0,$ which are valid to all orders. In the case of $x=0$, these formulae reduce to the asymptotic expressions recently obtained for the Riemann zeta function, which include the classical results of Siegel as a particular case.
\end{abstract}









\section{Introduction}
The \textit{Hurwitz zeta function} $\zeta(x,s)$ is a two-variable generalisation of the Riemann zeta function, defined by \[\zeta(x,s)\coloneqq\sum_{n=0}^{\infty}(n+x)^{-s}\;,\;\;\mathrm{Re}(x)>0,s=\sigma+it,\sigma>1,t\in\mathbb{R},\] and defined for all $s\in\mathbb{C}$ by analytic continuation. The \textit{modified Hurwitz zeta function} $\zeta_1(x,s)$ is a variant of the Hurwitz zeta function, defined by \[\zeta_1(x,s)\coloneqq\sum_{n=1}^{\infty}(n+x)^{-s}\;,\;\;\mathrm{Re}(x)>-1,s=\sigma+it,\sigma>1,t\in\mathbb{R},\] and again defined for all $s\in\mathbb{C}$ by analytic continuation. It is clear that these two functions are related by the simple formula \[\zeta(x,s)=\frac{1}{x^s}+\zeta_1(x,s)\;,\;\;\mathrm{Re}(x)>0,s\in\mathbb{C},\] and that for $x=0$ the modified Hurwitz function reduces to the Riemann zeta function: \[\zeta(s)=\zeta_1(x,s).\]

The following asymptotic formula for $\zeta(s)$, proved in e.g. Theorem 4.15 of \cite{titchmarsh}, is known as the \textit{approximate functional equation}:
\begin{equation}
\label{AFE}
\zeta(s)=\sum_{n\leq x}\frac{1}{n^s}+\chi(s)\sum_{n\leq y}\frac{1}{n^{1-s}}+O\Big(x^{-\sigma}+|t|^{\tfrac{1}{2}-\sigma}y^{\sigma-1}\Big),
\end{equation}
where \[xy=\tfrac{t}{2\pi},0<\sigma<1,t\rightarrow\infty,\] and the entire function $\chi(s)$ is defined by
\begin{equation}
\label{chi}
\chi(s)\coloneqq\tfrac{(2\pi)^s}{\pi}\Gamma(1-s)\sin\big(\tfrac{\pi s}{2}\big),\;s\in\mathbb{C}.
\end{equation}
(Throughout this paper, $\Gamma(s)$ denotes the gamma function of a complex variable $s$ and $\lfloor k\rfloor$ denotes the floor function of a real number $k$.)

The analogous formula for the modified Hurwitz function is the following asymptotic expression, proved in e.g. \cite{rane}:
\begin{equation}
\label{AFE-Hurwitz}
\zeta_1(\alpha,s)=\sum_{n=1}^{\lfloor x-\alpha\rfloor}(n+\alpha)^{-s}+\chi(s)\sum_{n=1}^{y}\tfrac{\sin\big(\tfrac{\pi s}{2}+2\pi n\alpha\big)}{\sin\big(\tfrac{\pi s}{2}\big)}n^{s-1}+O\big(x^{-\sigma}\log(y+2)+x^{1-\sigma}t^{-1/2}\big),
\end{equation}
where \[xy=\tfrac{t}{2\pi},0<\sigma<1,0<\alpha\leq1,t\rightarrow\infty.\]

Siegel, in his classical paper \cite{siegel} following Riemann's unpublished notes, found expressions for the error term in (\ref{AFE}) to all orders for the important particular case $x=y=\sqrt{\tfrac{t}{2\pi}}$. In \cite{fokas-lenells}, formulae analogous to those of Siegel were presented for any $x,y$ satisfying $xy=\tfrac{t}{2\pi}$. The starting point of the analysis of \cite{fokas-lenells} was the following exact formula, proved in Theorem 2.1 of \cite{fokas-lenells}:
\begin{equation}
\label{Riemann-starting-point}
\zeta(s)=\chi(s)\bigg[\sum_{n=1}^{\lfloor\eta/2\pi\rfloor}n^{s-1}+\frac{1}{(2\pi)^s}\bigg(-\frac{\eta^s}{s}+e^{i\pi s/2}\int_{-i\eta}^{\infty e^{i\phi_1}}\frac{z^{s-1}}{e^z-1}\,\mathrm{d}z+e^{-i\pi s/2}\int_{i\eta}^{\infty e^{i\phi_2}}\frac{z^{s-1}}{e^z-1}\,\mathrm{d}z\bigg)\bigg],
\end{equation}
valid for \[0<\eta<\infty,-\tfrac{\pi}{2}<\phi_1,\phi_2<\tfrac{\pi}{2},s\in\mathbb{C}.\]

The existence of the additional parameter $x$ occurring in $\zeta_1(x,s)$ leads to interesting results which do not have analogues for $\zeta(s)$; see for example \cite{wang}, \cite{balasubramanian}, \cite{andersson}, \cite{katsurada}, \cite{mezo}, and p. 73 in \cite{davenport}. In this paper, we present analogous results with those of \cite{fokas-lenells} for the modified Hurwitz function, obtaining asymptotics to all orders for the error term in (\ref{AFE-Hurwitz}). Our starting point is the following exact formula, which is proved in section 2:
\begin{multline}
\label{Hurwitz-starting-point}
\zeta_1(x,s)=\chi(s)\Bigg(\sum_{m=1}^{\lfloor\eta/2\pi\rfloor}e^{-2\pi imx}m^{s-1}-\tfrac{e^{-i\pi s/2}}{(2\pi)^s}\int_{\hat{C}^0_{\eta}}\frac{e^{(1+x)z}-e^{-xz}}{1-e^z}z^{s-1}\,\mathrm{d}z \\ +\tfrac{e^{i\pi s/2}}{(2\pi)^s}\int_{-i\eta}^{\infty e^{i\phi_2}}\frac{e^{-(1+x)z}}{1-e^{-z}}z^{s-1}\,\mathrm{d}z+\tfrac{e^{-i\pi s/2}}{(2\pi)^s}\int_{i\eta}^{\infty e^{i\phi_1}}\frac{e^{-(1+x)z}}{1-e^{-z}}z^{s-1}\,\mathrm{d}z,
\end{multline}
valid for \[0<\eta<\infty,-\tfrac{\pi}{2}<\phi_1,\phi_2<\tfrac{\pi}{2},0<\sigma\leq1,0<t<\infty,0<x<\infty.\]  We analyse this using an integration by parts method, as seen in \cite{miller}, and eventually derive an expression for the large-$t$ asymptotics of $\zeta_1(x,s)$ to all orders.

We note the following comparisons between our analysis and the analysis in \cite{fokas-lenells}.
\begin{enumerate}
\item Equation (\ref{Riemann-starting-point}) suggests separate analysis for the cases $t<\eta,t=\eta,t>\eta$. These three cases were indeed analysed separately in \cite{fokas-lenells}, but in our approach, we present a unified treatment. Our analysis requires a certain condition to be placed on $\eta$, but this condition is not very restrictive.
\item The asymptotic estimation of certain integrals appearing in \cite{fokas-lenells} led to their analysis via the stationary point technique. Here, by rewriting such integrals in terms of integrals which can be computed explicitly and integrals which do \textit{not} include stationary points, we have avoided the stationary point analysis.
\item The representations presented in \cite{fokas-lenells} involve a finite series for the case of $\eta<t$ but an infinite series for the case of $\eta\geq t$. Since our approach for all values of $\eta$ is analogous to that used in \cite{fokas-lenells} for the case of $\eta\geq t$, we first derive a representation which involves an infinite series. However, we are then able to replace this infinite series by a finite one, some of whose upper bounds depend on $\eta$. Thus, our final result is analogous to that of \cite{fokas-lenells} in the case of $\eta<t$, since it is a finite series, but it is less uniform in the sense that the length of this finite series depends on $\eta$.
\end{enumerate}

This paper is organised as follows: in section 2 we derive equation (\ref{Hurwitz-starting-point}); in sections 3, 4, and 5 we present the asymptotic analysis to all orders of the first, second, and third integrals in the RHS of (\ref{Hurwitz-starting-point}); and in section 6 we derive the main results. In the last section, we also show that in the case of $x=0$, our results are consistent with the formulae for the Riemann zeta function obtained in \cite{fokas-lenells}.

\section{The modified Hurwitz zeta function}
\begin{defn}
\label{Hankel}
The Hankel contour $H_{\alpha}$ is defined by the following three components:
\begin{align*}
L_1&=\{\alpha e^{i\theta}:\tfrac{\pi}{2}<\theta<\pi\}\cup\{re^{i\pi}:\alpha<r<\infty\}, \\
L_2&=\{re^{-i\pi}:\alpha<r<\infty\}\cup\{\alpha e^{i\theta}:-\pi<\theta<-\tfrac{\pi}{2}\}, \\
L_3&=\{\alpha e^{i\theta}:-\tfrac{\pi}{2}<\theta<\tfrac{\pi}{2}\},
\end{align*}
where $\alpha$ is a constant with $0<\alpha<2\pi$. We define $H_{\alpha}$ to be the union of the $L_j$, as shown in Figure \ref{fig1}.
\end{defn}

\begin{figure}
\centering
\includegraphics[width=0.6\linewidth]{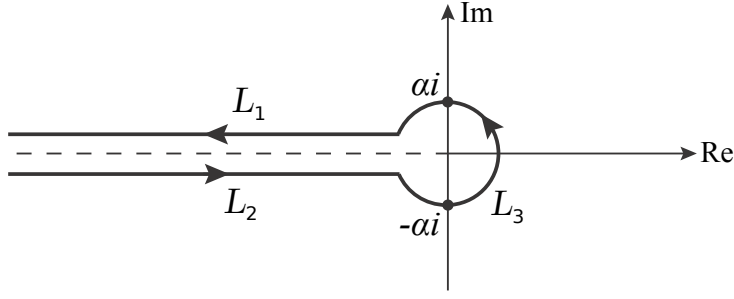}
\caption{The contours $L_1,L_2,L_3$ which together form the Hankel contour $H_{\alpha}$}
\label{fig1}
\end{figure}

\begin{lem}
\label{Hurwitz-AC}
The meromorphic continuation of the modified Hurwitz zeta function to all $s\in\mathbb{C}$ is given by
\begin{equation}
\label{Hurwitz-AC-formula}
\zeta_1(x,s)=\frac{\Gamma(1-s)}{2\pi i}\int_{H_{\alpha}}\frac{e^{xz}z^{s-1}}{e^{-z}-1}\,\mathrm{d}z,
\end{equation}
for $\mathrm{Re}(x)>-1$, where $H_{\alpha}$ is the Hankel contour defined by Definition \ref{Hankel}.
\end{lem}

\begin{proof}
For any $n\in\mathbb{N}$ and $s,x\in\mathbb{C}$ with $\mathrm{Re}(s)>1$ and $\mathrm{Re}(x)>-1$, we have \[\frac{1}{(n+x)^s}=\frac{1}{\Gamma(s)}\int_0^{\infty}e^{-(n+x)z}z^{s-1}\,\mathrm{d}z.\] Taking the sum over all $n\in\mathbb{N}$, and using the fact that $\sum_{n=1}^{\infty}e^{-nz}=\tfrac{1}{e^z-1}$ is a locally uniformly convergent series for $\mathrm{Re}(z)>0$, we find
\begin{equation}
\label{Hurwitz-integral}
\zeta_1(x,s)=\frac{1}{\Gamma(s)}\int_0^{\infty}e^{-xz}z^{s-1}\sum_{n=1}^{\infty}e^{-nz}\,\mathrm{d}z=\frac{1}{\Gamma(s)}\int_0^{\infty}\frac{e^{-xz}z^{s-1}}{e^z-1}\,\mathrm{d}z
\end{equation}
for $\mathrm{Re}(s)>1$ and $\mathrm{Re}(x)>-1$. In what follows we shall show that the right-hand sides of equations (\ref{Hurwitz-AC-formula}) and (\ref{Hurwitz-integral}) are identical and that the former is meromorphic for all $s,x\in\mathbb{C}$ with $\mathrm{Re}(x)>-1$; this will suffice to establish the required result.

For $\mathrm{Re}(s)>1$, we can let $\alpha\rightarrow0$ in the Hankel-contour formula, so that the integral around the curved part $L_3$ of the contour can be computed using Cauchy's theorem:
\begin{equation*}
\int_{L_3}\frac{e^{xz}z^{s-1}}{e^{-z}-1}\,\mathrm{d}z=\lim_{\alpha\rightarrow0}\Bigg(\int_{-\pi/2}^{\pi/2}\frac{e^{xz}z^s}{e^{-z}-1}\bigg|_{z=\alpha e^{i\theta}}\,\mathrm{d}\theta\Bigg)=\lim_{\alpha\rightarrow0}\pi\Big(\frac{z^s}{-z}\Big)\bigg|_{z=\alpha e^{i\theta}}=0.
\end{equation*}
Hence, the Hankel contour integral expression yields:
\begin{align*}
\frac{\Gamma(1-s)}{2\pi i}\int_{H_{\alpha}}\frac{e^{xz}z^{s-1}}{e^{-z}-1}\,\mathrm{d}z&=\frac{\Gamma(1-s)}{2\pi i}\bigg(\int_0^{\infty}\frac{e^{-xu}u^{s-1}e^{i\pi s}}{e^{u}-1}\,\mathrm{d}u-\int_0^{\infty}\frac{e^{-xu}u^{s-1}e^{-i\pi s}}{e^{u}-1}\,\mathrm{d}u\bigg) \\
&=\frac{\Gamma(1-s)\sin(\pi s)}{\pi}\int_0^{\infty}\frac{e^{-xu}u^{s-1}}{e^{u}-1}\,\mathrm{d}u \\
&=\frac{1}{\Gamma(s)}\int_0^{\infty}\frac{e^{-xz}z^{s-1}}{e^{z}-1}\,\mathrm{d}z,
\end{align*}
where we have used the substitutions $z=e^{i\pi}u$ along $L_1$ and $z=e^{-i\pi}u$ along $L_2$, and have replaced the dummy variable $u$ by $z$ in the final line.

Thus, we have proved that (\ref{Hurwitz-AC-formula}) holds as an identity for $\mathrm{Re}(s)>1$ and $\mathrm{Re}(x)>-1$. Also, the right-hand side of (\ref{Hurwitz-AC-formula}) is analytic for $\mathrm{Re}(x)>-1$ and all $s\in\mathbb{C}\backslash\mathbb{N}$, since the integrand is finite along the contour and entire in $x$ and $s$.
\end{proof}

\begin{lem}
\label{Hurwitz-alpha}
If $s=\sigma+it$ is a complex variable with $\sigma,t\in\mathbb{R}$ and $\sigma>0$, and $x$ is a real variable with $0<x<\infty$, then the modified Hurwitz zeta function $\zeta_1(x,s)$ can be expressed as
\begin{multline}
\label{Hurwitz-alpha-formula}
\zeta_1(x,s)=\frac{\chi(s)}{(2\pi)^s}\Bigg(\int_0^{\alpha}\frac{e^{i(1+x)u}-e^{-ixu}}{1-e^{iu}}u^{s-1}\,\mathrm{d}u \\
+e^{i\pi s/2}\int_{-i\alpha}^{\infty e^{i\phi_2}}\frac{e^{-(1+x)z}}{1-e^{-z}}z^{s-1}\,\mathrm{d}z+e^{-i\pi s/2}\int_{i\alpha}^{\infty e^{i\phi_1}}\frac{e^{-(1+x)z}}{1-e^{-z}}z^{s-1}\,\mathrm{d}z\Bigg),
\end{multline}
for any given $\alpha,\phi_1,\phi_2$ with $0<\alpha<2\pi$ and $-\tfrac{\pi}{2}<\phi_1,\phi_2<\tfrac{\pi}{2}$.
\end{lem}

\begin{proof}
We start with the expression (\ref{Hurwitz-AC-formula}) for $\zeta_1(x,s)$, and split the Hankel contour into the three parts $L_1$, $L_2$, $L_3$ defined in Definition \ref{Hankel}.

Firstly, by using Cauchy's theorem and then substituting $z=e^{i\pi}u,z=e^{-i\pi}u$ respectively, the integrals along $L_1$ and $L_2$ become:
\begin{align*}
\int_{L_1}\frac{e^{xz}z^{s-1}}{e^{-z}-1}\,\mathrm{d}z&=\int_{i\alpha}^{\infty e^{i\pi}}\frac{e^{(1+x)z}z^{s-1}}{1-e^{z}}\,\mathrm{d}z=e^{i\pi s}\int_{-i\alpha}^{\infty}\frac{e^{-(1+x)u}}{1-e^{-u}}u^{s-1}\,\mathrm{d}u; \\
\int_{L_2}\frac{e^{xz}z^{s-1}}{e^{-z}-1}\,\mathrm{d}z&=-\int_{-i\alpha}^{\infty e^{-i\pi}}\frac{e^{(1+x)z}z^{s-1}}{1-e^{z}}\,\mathrm{d}z=-e^{-i\pi s}\int_{i\alpha}^{\infty}\frac{e^{-(1+x)u}}{1-e^{-u}}u^{s-1}\,\mathrm{d}u.
\end{align*}

For the integral along $L_3$, we split the integrand as follows:
\begin{equation*}
\int_{L_3}\frac{e^{xz}z^{s-1}}{e^{-z}-1}\,\mathrm{d}z=-\int_{L_3}\frac{e^{-(1+x)z}}{1-e^{-z}}z^{s-1}\,\mathrm{d}z+\int_{L_3}\frac{e^{-(1+x)z}-e^{xz}}{1-e^{-z}}z^{s-1}\,\mathrm{d}z.
\end{equation*}
By Cauchy's theorem, the first of these integrals can be written as $\int_{L_3}=\int_{-i\alpha}^{\infty}-\int_{i\alpha}^{\infty}$. The integrand of the second integral behaves like $-(1+2x)z^{s-1}$ for $z$ close to $0$, so the integral of this function is finite even around $z=0$ (since we have assumed $\mathrm{Re}(s)>0$). This means the contour of integration can be deformed to the straight line-segment from $-i\alpha$ to $i\alpha$, and the integral can be simplified as follows:
\begin{align*}
&\int_{-i\alpha}^0\frac{e^{-(1+x)z}-e^{xz}}{1-e^{-z}}z^{s-1}\,\mathrm{d}z+\int_0^{i\alpha}\frac{e^{-(1+x)z}-e^{xz}}{1-e^{-z}}z^{s-1}\,\mathrm{d}z \\
=&-e^{-i\pi s/2}\int_0^{\alpha}\frac{e^{i(1+x)u}-e^{-ixu}}{1-e^{iu}}u^{s-1}\,\mathrm{d}u+e^{i\pi s/2}\int_0^{\alpha}\frac{e^{-i(1+x)u}-e^{ixu}}{1-e^{-iu}}u^{s-1}\,\mathrm{d}u \\
=&-e^{-i\pi s/2}\int_0^{\alpha}\frac{e^{i(1+x)u}-e^{-ixu}}{1-e^{iu}}u^{s-1}\,\mathrm{d}u+e^{i\pi s/2}\int_0^{\alpha}\frac{e^{-ixu}-e^{i(1+x)u}}{e^{iu}-1}u^{s-1}\,\mathrm{d}u \\
=&2i\sin(\tfrac{\pi s}{2})\int_0^{\alpha}\frac{e^{i(1+x)u}-e^{-ixu}}{1-e^{iu}}u^{s-1}\,\mathrm{d}u.
\end{align*}

Summing up the expressions derived for the integrals along $L_1$, $L_2$, and $L_3$, we find that (\ref{Hurwitz-AC-formula}) yields:
\begin{align*}
\begin{split}
\zeta_1(x,s)&=\frac{\Gamma(1-s)}{2\pi i}\bigg(e^{i\pi s}\int_{-i\alpha}^{\infty}\frac{e^{-(1+x)u}}{1-e^{-u}}u^{s-1}\,\mathrm{d}u-e^{-i\pi s}\int_{i\alpha}^{\infty}\frac{e^{-(1+x)u}}{1-e^{-u}}u^{s-1}\,\mathrm{d}u \\
&\hspace{3cm}+\int_{i\alpha}^{\infty}\frac{e^{-(1+x)z}}{1-e^{-z}}z^{s-1}\,\mathrm{d}z-\int_{-i\alpha}^{\infty}\frac{e^{-(1+x)z}}{1-e^{-z}}z^{s-1}\,\mathrm{d}z \\
&\hspace{5cm}+2i\sin(\tfrac{\pi s}{2})\int_0^{\alpha}\frac{e^{i(1+x)u}-e^{-ixu}}{1-e^{iu}}u^{s-1}\,\mathrm{d}u\bigg)
\end{split} \\
\begin{split}
&=\tfrac{\Gamma(1-s)}{\pi}\sin(\tfrac{\pi s}{2})\bigg(e^{i\pi s/2}\int_{-i\alpha}^{\infty}\frac{e^{-(1+x)z}}{1-e^{-z}}z^{s-1}\,\mathrm{d}z+e^{-i\pi s/2}\int_{i\alpha}^{\infty}\frac{e^{-(1+x)z}}{1-e^{-z}}z^{s-1}\,\mathrm{d}z \\
&\hspace{8cm}+\int_0^{\alpha}\frac{e^{i(1+x)u}-e^{-ixu}}{1-e^{iu}}u^{s-1}\,\mathrm{d}u\bigg).
\end{split}
\end{align*}
In the final expression above, the integrands of the first two integrals decay exponentially as $z$ tends to infinity in the right half plane. Thus, by Cauchy's theorem, the upper limits of these integrals can be replaced by $\infty e^{i\phi_1}$ and $\infty e^{i\phi_2}$ respectively for any $\phi_1,\phi_2\in\big(\tfrac{-\pi}{2},\tfrac{\pi}{2}\big)$. The expression outside the large parentheses is precisely $\tfrac{\chi(s)}{(2\pi)^s}$, so the result follows.
\end{proof}

\begin{defn}
\label{C-curves}
For $a,b\in\mathbb{R}$ with $a<b$, the curves $C_a^b$ and $\hat{C}_b^a$ are defined as follows:
\begin{align*}
C_a^b&=\bigg\{\frac{i(a+b)}{2}+\frac{b-a}{2}e^{i\theta}\,:\,\theta\in\Big(-\frac{\pi}{2},\frac{\pi}{2}\Big)\bigg\}; \\
\hat{C}_b^a&=\bigg\{\frac{i(a+b)}{2}+\frac{b-a}{2}e^{i\theta}\,:\,\theta\in\Big(-\pi,-\frac{\pi}{2}\Big)\cup\Big(\frac{\pi}{2},\pi\Big)\bigg\}.
\end{align*}
In other words, $C_a^b$ is the semicircular contour from $ia$ to $ib$ passing upwards through the right half plane, while $\hat{C}_b^a$ is the semicircular contour from $ib$ to $ia$ passing downwards through the left half plane. The two together form a full circular contour, as shown in Figure \ref{fig2}.
\end{defn}

\begin{figure}
\centering
\includegraphics[width=0.6\linewidth]{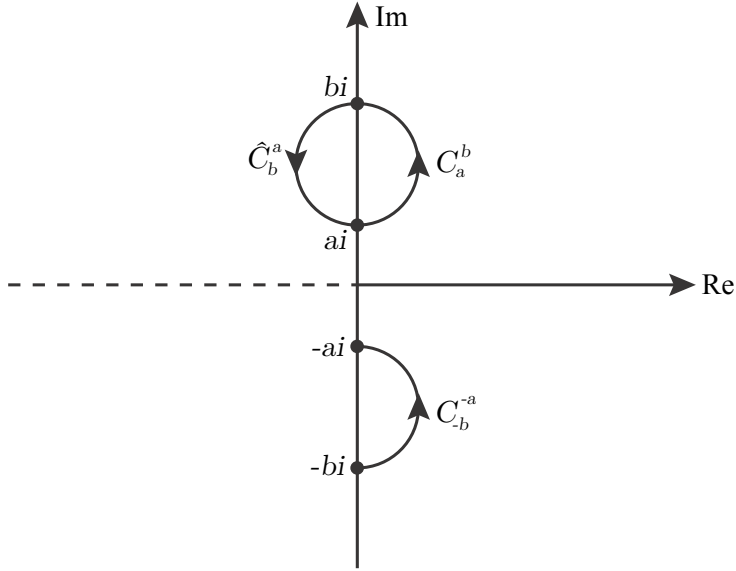}
\caption{The contours $C_a^b$, $\hat{C}_b^a$, and $C_{-b}^{-a}$}
\label{fig2}
\end{figure}

\begin{thm}
\label{Hurwitz-eta}
If $s=\sigma+it$ is a complex variable with $\sigma,t\in\mathbb{R},0<\sigma\leq1,0<t<\infty$, and $x$ is a real variable with $0<x<\infty$, then the modified Hurwitz zeta function $\zeta_1(x,s)$ can be expressed as
\begin{equation}
\label{Hurwitz-eta-formula}
\zeta_1(x,s)=\chi(s)\Bigg(\sum_{m=1}^{\lfloor\eta/2\pi\rfloor}e^{-2\pi imx}m^{s-1}-G_B(\sigma,t;\eta;x)+G_L(\sigma,t;\eta;x)+G_U(\sigma,t;\eta;x)\Bigg),
\end{equation}
for any given $\eta,\phi_1,\phi_2$ with $0<\eta<\infty$ and $-\frac{\pi}{2}<\phi_1,\phi_2<\frac{\pi}{2}$, where
\begin{align}
\label{GB-defn}
G_B(\sigma,t;\eta;x)&\coloneqq\tfrac{e^{-i\pi s/2}}{(2\pi)^s}\int_{\hat{C}^0_{\eta}}\frac{e^{(1+x)z}-e^{-xz}}{1-e^z}z^{s-1}\,\mathrm{d}z, \\
\label{GL-defn}
G_L(\sigma,t;\eta;x)&\coloneqq\tfrac{e^{i\pi s/2}}{(2\pi)^s}\int_{-i\eta}^{\infty e^{i\phi_2}}\frac{e^{-(1+x)z}}{1-e^{-z}}z^{s-1}\,\mathrm{d}z, \\
\label{GU-defn}
G_U(\sigma,t;\eta;x)&\coloneqq\tfrac{e^{-i\pi s/2}}{(2\pi)^s}\int_{i\eta}^{\infty e^{i\phi_1}}\frac{e^{-(1+x)z}}{1-e^{-z}}z^{s-1}\,\mathrm{d}z.
\end{align}
\end{thm}

\begin{proof}
We start with the result of Lemma \ref{Hurwitz-alpha}. By Cauchy's theorem, the contours of integrations in the second and third integrals in (\ref{Hurwitz-alpha-formula}) can be deformed so as to run first from $\pm i\alpha$ to $\pm i\eta$ and then out to infinity: \[\int_{-i\alpha}^{\infty e^{i\phi_2}}=-\int_{C_{-\eta}^{-\alpha}}+\int_{-i\eta}^{\infty e^{i\phi_2}}\,\text{ and }\,\int_{i\alpha}^{\infty e^{i\phi_1}}=\int_{C_{\alpha}^{\eta}}+\int_{i\eta}^{\infty e^{i\phi_1}}.\] Hence, the sum of the last two terms in (\ref{Hurwitz-alpha-formula}) is equal to the sum of $(2\pi)^sG_L(\sigma,t;\eta;x)$ and $(2\pi)^sG_U(\sigma,t;\eta;x)$ with the following expression:
\begin{equation}
\label{C-curve-integrals}
-e^{i\pi s/2}\int_{C_{-\eta}^{-\alpha}}\frac{e^{-(1+x)z}}{1-e^{-z}}z^{s-1}\,\mathrm{d}z+e^{-i\pi s/2}\int_{C_{\alpha}^{\eta}}\frac{e^{-(1+x)z}}{1-e^{-z}}z^{s-1}\,\mathrm{d}z.
\end{equation}
The first term of this expression, after the substitution $z=ue^{-i\pi}$, becomes
\begin{align*}
&-e^{-i\pi s/2}\int_{\hat{C}_{\eta}^{\alpha}}\frac{e^{(1+x)u}}{1-e^{u}}u^{s-1}\,\mathrm{d}u \\
=\,&e^{-i\pi s/2}\int_{\hat{C}_{\eta}^{\alpha}}\frac{e^{xu}}{1-e^{-u}}u^{s-1}\,\mathrm{d}u \\
=\,&e^{-i\pi s/2}\int_{\hat{C}_{\eta}^{\alpha}}\frac{e^{-(1+x)z}}{1-e^{-z}}z^{s-1}\,\mathrm{d}z+e^{-i\pi s/2}\int_{\hat{C}_{\eta}^{\alpha}}\frac{e^{xz}-e^{-(1+x)z}}{1-e^{-z}}z^{s-1}\,\mathrm{d}z.
\end{align*}
So the expression (\ref{C-curve-integrals}) can be rewritten as: \[e^{-i\pi s/2}\int_{C'}\frac{e^{-(1+x)z}}{1-e^{-z}}z^{s-1}\,\mathrm{d}z+e^{-i\pi s/2}\int_{\hat{C}_{\eta}^{\alpha}}\frac{e^{xz}-e^{-(1+x)z}}{1-e^{-z}}z^{s-1}\,\mathrm{d}z,\] where $C'$ is the circle with centre $\tfrac{\alpha+\eta}{2}$ formed by combining the two semicircles $C_{\alpha}^{\eta}$ and $\hat{C}_{\eta}^{\alpha}$. By the residue theorem, we can compute this circular integral explicitly:
\begin{equation*}
\int_{C'}\frac{e^{-(1+x)z}}{1-e^{-z}}z^{s-1}\,\mathrm{d}z=\sum_{m=1}^{\lfloor\eta/2\pi\rfloor}2\pi i\mathrm{Res}_{2\pi mi}\bigg(\frac{e^{-(1+x)z}}{1-e^{-z}}z^{s-1}\bigg)=\sum_{m=1}^{\lfloor\eta/2\pi\rfloor}e^{-2\pi imx}m^{s-1}(2\pi i)^s.
\end{equation*}
So the expression (\ref{C-curve-integrals}) can be rewritten as:
\begin{align*}
&(2\pi)^s\sum_{m=1}^{\lfloor\eta/2\pi\rfloor}e^{-2\pi imx}m^{s-1}+e^{-i\pi s/2}\int_{\hat{C}_{\eta}^{\alpha}}\frac{e^{xz}-e^{-(1+x)z}}{1-e^{-z}}z^{s-1}\,\mathrm{d}z \\
=\,&(2\pi)^s\sum_{m=1}^{\lfloor\eta/2\pi\rfloor}e^{-2\pi imx}m^{s-1}+e^{-i\pi s/2}\int_{\hat{C}_{\eta}^0}\frac{e^{xz}-e^{-(1+x)z}}{1-e^{-z}}z^{s-1}\,\mathrm{d}z+e^{-i\pi s/2}\int_0^{i\alpha}\frac{e^{xz}-e^{-(1+x)z}}{1-e^{-z}}z^{s-1}\,\mathrm{d}z.
\end{align*}
The last term of this expression, after substituting $z=iu$, becomes exactly minus the first integral in (\ref{Hurwitz-alpha-formula}). Hence, starting from the formula (\ref{Hurwitz-alpha-formula}) for $\zeta_1(x,s)$, we find:
\begin{align*}
\zeta_1(x,s)&=\frac{\chi(s)}{(2\pi)^s}\bigg(\int_0^{\alpha}\frac{e^{i(1+x)u}-e^{-ixu}}{1-e^{iu}}u^{s-1}\,\mathrm{d}u+(2\pi)^sG_L(\sigma,t;\eta;x)+(2\pi)^sG_U(\sigma,t;\eta;x)+(\ref{C-curve-integrals})\bigg) \\
\begin{split}
&=\frac{\chi(s)}{(2\pi)^s}\bigg((2\pi)^sG_L(\sigma,t;\eta;x)+(2\pi)^sG_U(\sigma,t;\eta;x)+(2\pi)^s\sum_{m=1}^{\lfloor\eta/2\pi\rfloor}e^{-2\pi imx}m^{s-1} \\
&\hspace{7cm}+e^{-i\pi s/2}\int_{\hat{C}_{\eta}^0}\frac{e^{xz}-e^{-(1+x)z}}{1-e^{-z}}z^{s-1}\,\mathrm{d}z\bigg)
\end{split} \\
&=\chi(s)\bigg(\sum_{m=1}^{\lfloor\eta/2\pi\rfloor}e^{-2\pi imx}m^{s-1}+G_L(\sigma,t;\eta;x)+G_U(\sigma,t;\eta;x)-G_B(\sigma,t;\eta;x)\bigg),
\end{align*}
as required.
\end{proof}

\begin{remark}
\normalfont For the particular case of $x=0$, we find \[G_B(\sigma,t;\eta;0)=\tfrac{e^{-i\pi s/2}}{(2\pi)^s}\int_{\hat{C}^0_{\eta}}\big(-z^{s-1}\big)\,\mathrm{d}z=\frac{1}{(2\pi)^s}\Big(-\frac{\eta^s}{s}\Big),\] and therefore the identity (\ref{Hurwitz-eta-formula}) reduces, as expected, to the formula (\ref{Riemann-starting-point}) proved in \cite{fokas-lenells}.
\end{remark}

In the remainder of this paper, we shall construct the large-$t$ asymptotics of each of $G_L,G_U,G_B$, and hence derive a large-$t$ asymptotic formula for $\zeta_1(x,s)$.

\begin{lem}
\label{D_N}
The function $D_N$ defined by
\begin{equation}
\label{DN-defn}
D_N(z;\xi;\sigma,t)\coloneqq\big(\tfrac{\mathrm{d}}{\mathrm{d}z}\cdot\tfrac{1}{\xi-\frac{it}{z}}\big)^N(z^{\sigma-1})
\end{equation}
can be expressed in the following form for any $N\geq0$:
\begin{equation}
\label{DN-series}
D_N=\sum_{b=0}^N\sum_{c=0}^NA_{bc}^{(N)}\Bigg(\frac{t^b\xi^{N-b}\sigma^cz^{N-b}}{(\xi z-it)^{2N}}\Bigg)z^{\sigma-1},
\end{equation}
where $A_{bc}^{(N)}$ is a Gaussian integer with absolute value $\leq(2N-1)!!\coloneqq(1)(3)(5)\dots(2N-3)(2N-1)$ for each $b,c$.
\end{lem}

\begin{proof}
Following the argument of \cite{fokas-lenells}, we proceed by induction on $N$.

In the base case $N=0$, we must also have $b=c=0$ and so the expression (\ref{DN-series}) reduces to \[D_N=A_{00}^{(0)}\Bigg(\frac{t^0\xi^{0}\sigma^0z^{0}}{(\xi z-it)^{0}}\Bigg)z^{\sigma-1}=A_{00}^{(0)}z^{\sigma-1}.\] By (\ref{DN-defn}), this is valid with $A_{00}^{(0)}=1=(-1)!!$. (It makes sense to define $(-1)!!=1$ in the same way as we ordinarily define $0!=1$, because $(2N+1)!!=(2N-1)!!(2N+1)$ for all $N$ and $1!!=1$.)

Now assume that $D_N$ can be written in the form (\ref{DN-series}) for some fixed $N\geq0$, and consider $D_{N+1}$. Using the definition (\ref{DN-defn}), we have:
\begin{align*}
D_{N+1}&=\frac{\mathrm{d}}{\mathrm{d}z}\bigg(\frac{D_N}{\xi-\tfrac{it}{z}}\bigg)=\frac{\mathrm{d}}{\mathrm{d}z}\bigg(\frac{zD_N}{\xi z-it}\bigg) \\
&=\frac{\mathrm{d}}{\mathrm{d}z}\Bigg(\sum_{b=0}^N\sum_{c=0}^NA_{bc}^{(N)}\Bigg(\frac{t^b\xi^{N-b}\sigma^cz^{N+1-b}}{(\xi z-it)^{2N+1}}\Bigg)z^{\sigma-1}\Bigg) \\
&=\sum_{b=0}^N\sum_{c=0}^NA_{bc}^{(N)}\frac{\big[t^b\xi^{N-b}\sigma^c(N+\sigma-b)z^{N+\sigma-b-1}\big](\xi z-it)-\big[t^b\xi^{N-b}\sigma^cz^{N+\sigma-b}\big](2N+1)\xi}{(\xi z-it)^{2N+2}} \\
&=\sum_{b=0}^N\sum_{c=0}^NA_{bc}^{(N)}\frac{t^b\xi^{N+1-b}\sigma^c(-N-1+\sigma-b)z^{N+\sigma-b}-it^{b+1}\xi^{N-b}\sigma^c(N+\sigma-b)z^{N+\sigma-b-1}}{(\xi z-it)^{2N+2}} \\
\begin{split}
&=\sum_{b=0}^N\sum_{c=0}^N\Bigg(-(N+1+b)A_{bc}^{(N)}\frac{t^b\xi^{N+1-b}\sigma^cz^{N+1-b}}{(\xi z-it)^{2N+2}}+A_{bc}^{(N)}\frac{t^b\xi^{N+1-b}\sigma^{c+1}z^{N+1-b}}{(\xi z-it)^{2N+2}} \\
&\hspace{3cm}-i(N-b)A_{bc}^{(N)}\frac{t^{b+1}\xi^{N-b}\sigma^cz^{N-b}}{(\xi z-it)^{2N+2}}-iA_{bc}^{(N)}\frac{t^{b+1}\xi^{N-b}\sigma^{c+1}z^{N-b}}{(\xi z-it)^{2N+2}}\Bigg)z^{\sigma-1}.
\end{split}
\end{align*}
Thus, setting the values of $A_{bc}^{(N+1)}$ as suggested by this expression, we obtain a formula for $D_{N+1}$ in the form of (\ref{DN-series}).
\end{proof}

\begin{assn}
\label{Assumption}
We shall fix $\epsilon>0$ and assume that the variable $\eta$ is never, for any integer $n$, within a factor of $1\pm\epsilon$ of the quantity $\tfrac{t}{x+n}$. In other words, we assume that \[\forall n\in\mathbb{Z},\;\text{ either }\;\eta>(1+\epsilon)\tfrac{t}{x+n}\;\text{ or }\;\eta<(1-\epsilon)\tfrac{t}{x+n}.\]
The above assumption can be rewritten as \[\mathrm{dist}\Big(x-\frac{t}{\eta},\mathbb{Z}\Big)>\frac{\epsilon t}{\eta},\] or equivalently as
\begin{equation}
\label{assumption-inequality}
\forall n\in\mathbb{Z},\;\big|(x+n)\eta-t\big|>\epsilon t.
\end{equation}
Note that to find out whether a given $\eta$ satisfies (\ref{assumption-inequality}), it suffices to check for the particular value of $n$ such that $|(x+n)\eta-t|$ is minimal, i.e. for $n=\lfloor\tfrac{t}{\eta}-x+\tfrac{1}{2}\rfloor$, the closest integer to $\tfrac{t}{\eta}-x$. This $n$ may be either a positive or negative integer, depending on the values of $x$, $t$, and $\eta$.
\end{assn}

\section{Asymptotics for $G_L$}

The series $\sum_{n=0}^{\infty}e^{-nz}=\frac{1}{1-e^{-z}}$ is locally uniformly convergent for $\mathrm{Re}(z)>0$, so we can interchange the series and integral to obtain
\begin{equation}
\label{GL-series}
G_L(\sigma,t;\eta;x)=\tfrac{e^{i\pi s/2}}{(2\pi)^s}\sum_{n=1}^{\infty}\int_{-i\eta}^{\infty e^{i\phi_2}}e^{-(x+n)z}z^{s-1}\,\mathrm{d}z.
\end{equation}
Repeatedly integrating by parts in the summand gives
\begin{multline}
\label{GL-summand-IbP}
\int_{-i\eta}^{\infty e^{i\phi_2}}e^{-(x+n)z}z^{s-1}\,\mathrm{d}z=\sum_{j=0}^{N-1}e^{-(x+n)z+it\log z}\Big(\tfrac{1}{x+n-\frac{it}{z}}\cdot\tfrac{\mathrm{d}}{\mathrm{d}z}\Big)^j\Big(\tfrac{z^{\sigma-1}}{x+n-\frac{it}{z}}\Big)\Bigg|_{z=-i\eta} \\ +\int_{-i\eta}^{\infty e^{i\phi_2}}e^{-(x+n)z+it\log z}D_N(z;x+n;\sigma,t)\,\mathrm{d}z
\end{multline}
for any $N\in\mathbb{N}$, where $D_N$ is defined by (\ref{DN-defn}).

In what follows, we shall take $\phi_2=0$, so that $z\in-i\eta+\mathbb{R}^+$.

\begin{lem}
\label{GL-DN-estimate}
$D_N$ can be uniformly estimated in either of the following ways, both valid for $\mathrm{Im}(z)<0$ and $\xi>0$:
\begin{align}
\label{GL-DN-estimate1}
D_N(z;\xi;\sigma,t)&=O\big((2N-1)!!(N+1)^2|z|^{\sigma-N-1}\xi^{-N}\big); \\
\label{GL-DN-estimate2}
D_N(z;\xi;\sigma,t)&=O\big((2N-1)!!(N+1)^2|z|^{\sigma-1}t^{-N}\big).
\end{align}
\end{lem}

\begin{proof}
By (\ref{DN-series}), we have the following two expressions for $D_N$:
\begin{align*}
D_N(z;\xi;\sigma,t)&=O\Bigg((2N-1)!!|z|^{\sigma-1}\sum_{b=0}^N\sum_{c=0}^N\Big|\tfrac{t}{\xi z}\Big|^b|\sigma|^c\Big|\tfrac{\xi z}{(\xi z-it)^2}\Big|^N\Bigg); \\
D_N(z;\xi;\sigma,t)&=O\Bigg((2N-1)!!|z|^{\sigma-1}\sum_{b=0}^N\sum_{c=0}^N\Big|\tfrac{\xi z}{t}\Big|^{N-b}|\sigma|^c\Big|\tfrac{t}{(\xi z-it)^2}\Big|^N\Bigg).
\end{align*}
Since $0<\sigma\leq1$ and $|\xi z-it|$ is greater than both $|\xi z|$ and $t$ (by our assumption on $z$ and the fact that $\xi$ and $t$ are positive reals), we can simplify these estimates as follows.

\textbf{Case 1:} $\boldsymbol{|\xi z|>t}.$

In this case,
\begin{align*}
D_N&=O\Bigg((2N-1)!!|z|^{\sigma-1}\sum_{b=0}^N\Big|\tfrac{t}{\xi z}\Big|^b(N+1)\Big|\tfrac{\xi z}{(\xi z)^2}\Big|^N\Bigg) \nonumber=O\Big((2N-1)!!(N+1)^2|z|^{\sigma-N-1}\xi ^{-N}\Big).
\end{align*}

\textbf{Case 2:} $\boldsymbol{|\xi z|<t}.$

In this case,
\begin{align*}
D_N&=O\Bigg((2N-1)!!|z|^{\sigma-1}\sum_{b=0}^N\Big|\tfrac{\xi z}{t}\Big|^{N-b}(N+1)\big|\tfrac{t}{t^2}\big|^N\Bigg) \nonumber=O\Big((2N-1)!!(N+1)^2|z|^{\sigma-1}t^{-N}\Big).
\end{align*}
In both cases, we have $D_N=O\Big((2N-1)!!(N+1)^2|z|^{\sigma-1}\max\big(\xi |z|,t\big)^{-N}\Big)$, from which both the estimates (\ref{GL-DN-estimate1}), (\ref{GL-DN-estimate2}) follow. Note also that in each case the bound is uniform in all variables: in fact, the $O$-constant can be taken to be $1$.
\end{proof}

\begin{lem}
\label{GL-estimate}
We have the following estimate for $G_L$, uniform in $\sigma$, $t$, $\eta$, $x$, and $N\geq1$:
\begin{multline*}
G_L(\sigma,t;\eta;x)=\frac{e^{i\pi\sigma/2}e^{it\log\eta}}{(2\pi)^s}\sum_{n=1}^{M}\sum_{j=0}^{N-1}e^{i(x+n)\eta}\Bigg[\Big(\tfrac{1}{x+n-\frac{it}{z}}\cdot\tfrac{\mathrm{d}}{\mathrm{d}z}\Big)^j\Big(\tfrac{z^{\sigma-1}}{x+n-\frac{it}{z}}\Big)\Bigg]_{z=-i\eta} \\ +O\Big((2N-1)!!(N+1)^2\eta^{\sigma-N-1}\Big),
\end{multline*}
where $M$ is a finite number depending only on $N$ and $\eta$.
\end{lem}

\begin{proof}
Employing (\ref{GL-series}) and (\ref{GL-summand-IbP}) (where we have set $\phi_2=0$), it suffices to estimate \[\sum_{n=1}^{\infty}\int_{-i\eta}^{\infty}e^{-(x+n)z+it\log z}D_N(z;x+n;\sigma,t)\,\mathrm{d}z,\] which can be achieved using Lemma \ref{GL-DN-estimate}. By equation (\ref{GL-DN-estimate1}), this expression is given by:
\begin{align*}
O&\Bigg(\sum_{n=1}^{\infty}\int_{-i\eta}^{\infty}e^{-(x+n)z+it\log z}(2N-1)!!(N+1)^2|z|^{\sigma-N-1}(x+n)^{-N}\,\mathrm{d}z\Bigg) \\
=O&\Bigg((2N-1)!!(N+1)^2\sum_{n=1}^{\infty}\int_{-i\eta}^{\infty}e^{-(x+n)\mathrm{Re}(z)}e^{\pi t/2}\eta^{\sigma-N-1}(x+n)^{-N}\,\mathrm{d}z\Bigg) \\
=O&\Bigg((2N-1)!!(N+1)^2e^{\pi t/2}\eta^{\sigma-N-1}\sum_{n=1}^{\infty}(x+n)^{-N}\int_{0}^{\infty}e^{-(x+n)u}\,\mathrm{d}z\Bigg) \\
=O&\Bigg((2N-1)!!(N+1)^2e^{\pi t/2}\eta^{\sigma-N-1}\sum_{n=1}^{\infty}(x+n)^{-N-1}\Bigg)=O\Big((2N-1)!!(N+1)^2e^{\pi t/2}\eta^{\sigma-N-1}\Big).
\end{align*}

Since $e^{i\pi s/2}=e^{i\pi\sigma/2}e^{-\pi t/2}$ and $e^{it\log(-i\eta)}=e^{it\log\eta}e^{\pi t/2}$, using the above estimate together with (\ref{GL-series}) and (\ref{GL-summand-IbP}) yields:
\begin{align*}
G_L(\sigma,t;\eta;x)&=\tfrac{e^{i\pi\sigma/2}e^{-\pi t/2}}{(2\pi)^s}\sum_{n=1}^{\infty}\int_{-i\eta}^{\infty e^{i\phi_2}}e^{-(x+n)z}z^{s-1}\,\mathrm{d}z \\
\begin{split}
&=\frac{e^{i\pi\sigma/2}e^{it\log\eta}}{(2\pi)^s}\sum_{n=1}^{\infty}\sum_{j=0}^{N-1}e^{i(x+n)\eta}\Bigg[\Big(\tfrac{1}{x+n-\frac{it}{z}}\cdot\tfrac{\mathrm{d}}{\mathrm{d}z}\Big)^j\Big(\tfrac{z^{\sigma-1}}{x+n-\frac{it}{z}}\Big)\Bigg]_{z=-i\eta} \\ &\hspace{3cm} +\frac{e^{i\pi\sigma/2}e^{-\pi t/2}}{(2\pi)^s}\sum_{n=1}^{\infty}\int_{-i\eta}^{\infty e^{i\phi_2}}e^{-(x+n)z+it\log z}D_N(z;x+n;\sigma,t)\,\mathrm{d}z
\end{split}\\
\begin{split}
&=\frac{e^{i\pi\sigma/2}e^{it\log\eta}}{(2\pi)^s}\sum_{n=1}^{\infty}\sum_{j=0}^{N-1}e^{i(x+n)\eta}\Bigg[\Big(\tfrac{1}{x+n-\frac{it}{z}}\cdot\tfrac{\mathrm{d}}{\mathrm{d}z}\Big)^j\Big(\tfrac{z^{\sigma-1}}{x+n-\frac{it}{z}}\Big)\Bigg]_{z=-i\eta} \\ &\hspace{7cm}+O\Big((2N-1)!!(N+1)^2\eta^{\sigma-N-1}\Big).
\end{split}
\end{align*}
The uniformity of the $O$-bound is inherited from Lemma \ref{GL-DN-estimate}. In order to derive the final result, we just need to find $M(N,\eta)$ large enough so that
\begin{equation}
\label{GL-M-tail}
\sum_{n=M+1}^{\infty}\sum_{j=0}^{N-1}e^{i(x+n)\eta}\Bigg[\Big(\tfrac{1}{x+n-\frac{it}{z}}\cdot\tfrac{\mathrm{d}}{\mathrm{d}z}\Big)^j\Big(\tfrac{z^{\sigma-1}}{x+n-\frac{it}{z}}\Big)\Bigg]_{z=-i\eta}=O\Big((2N-1)!!(N+1)^2\eta^{\sigma-N-1}\Big).
\end{equation}
Using the definition of $D_N$ and the bound (\ref{GL-DN-estimate1}) for $D_N$, we can estimate the left hand side of (\ref{GL-M-tail}) as follows:
\begin{align*}
&{\color{white}=}\sum_{n=M+1}^{\infty}\sum_{j=0}^{N-1}\frac{e^{i(x+n)\eta}}{x+n+\frac{t}{\eta}}D_j(-i\eta;x+n;\sigma,t) \\
&=\sum_{n=M+1}^{\infty}\frac{e^{i(x+n)\eta}(-i\eta)^{\sigma-1}}{x+n+\frac{t}{\eta}}+\sum_{n=M+1}^{\infty}\sum_{j=1}^{N-1}O\bigg(\frac{1}{x+n+\frac{t}{\eta}}(2j-1)!!(j+1)^2\eta^{\sigma-j-1}(x+n)^{-j}\bigg) \\
&=O\bigg(\eta^{\sigma-1}\sum_{n=M+1}^{\infty}\frac{e^{in\eta}}{x+n+\frac{t}{\eta}}\bigg)+\sum_{j=1}^{N-1}O\bigg((2j-1)!!(j+1)^2\eta^{\sigma-j-1}\sum_{n=M+1}^{\infty}(x+n)^{-j-1}\bigg) \\
&=O\bigg(\eta^{\sigma-1}\sum_{n=M+1}^{\infty}\frac{e^{in\eta}}{x+n}\bigg)+O\bigg((2N-3)!!N^2\sum_{j=1}^{N-1}\eta^{\sigma-j-1}\sum_{n=M+1}^{\infty}(x+n)^{-j-1}\bigg).
\end{align*}
All of the infinite series in this expression are convergent, so we can simply choose $M$ large enough so that \[\sum_{n=M+1}^{\infty}\frac{e^{in\eta}}{x+n}\leq\eta^{-N}\] and \[\sum_{n=M+1}^{\infty}(x+n)^{-j-1}\leq\eta^{j-N}\] for $j=1,2,\dots,N-1$. In the second of these inequalities, the left hand side is decreasing in $j$ while the right hand side is increasing in $j$, so we can simplify the conditions to
\begin{equation}
\label{GL-M-condition1}
\sum_{n=M+1}^{\infty}e^{in\eta}n^{-1}\leq\eta^{-N}
\end{equation}
and
\begin{equation}
\label{GL-M-condition2}
\sum_{n=M+1}^{\infty}n^{-2}\leq\eta^{1-N}.
\end{equation}
For any $M$ satisfying (\ref{GL-M-condition1}) and (\ref{GL-M-condition2}), we have the required bound (\ref{GL-M-tail}), and so the final result holds.
\end{proof}

\section{Asymptotics for $G_U$}
As before, the series $\sum_{n=0}^{\infty}e^{-nz}=\frac{1}{1-e^{-z}}$ is locally uniformly convergent for $\mathrm{Re}(z)>0$, so we can interchange the series and integral to obtain
\begin{equation}
\label{GU-series}
G_U(\sigma,t;\eta;x)=\tfrac{e^{-i\pi s/2}}{(2\pi)^s}\sum_{n=1}^{\infty}\int_{i\eta}^{\infty e^{i\phi_1}}e^{-(x+n)z}z^{s-1}\,\mathrm{d}z.
\end{equation}
Let us fix $\phi_1=\tfrac{\pi}{2}$, so that $z\in i[\eta,\infty)$. Now the integrand is $e^{-(x+n)z+it}z^{\sigma-1}$, which has a stationary point iff $-(x+n)+\tfrac{it}{z}=0$, i.e. at $z=\tfrac{it}{x+n}$. So there is a stationary point in the interval of integration iff \[\tfrac{it}{x+n}\in i[\eta,\infty), \text{ i.e. } \eta\leq\tfrac{t}{x+n}<\infty, \text{ i.e. } n\leq\tfrac{t}{\eta}-x.\]

This is the first place we need to use Assumption \ref{Assumption}. The inequality (\ref{assumption-inequality}) can be rearranged in terms of $n$, since its opposite statement rearranges as follows:
\begin{align*}
|(x+n)\eta-t|\geq\epsilon t&\Leftrightarrow(1-\epsilon)t\leq(x+n)\eta\leq(1+\epsilon)t \\ &\Leftrightarrow(1-\epsilon)\tfrac{t}{\eta}-x\leq n\leq(1+\epsilon)\tfrac{t}{\eta}-x.
\end{align*}
So we need to consider two separate cases, namely $n<(1-\epsilon)\tfrac{t}{\eta}-x$ and $n>(1+\epsilon)\tfrac{t}{\eta}-x$. In other words, the sum over $n$ appearing in (\ref{GU-series}) needs to be split into two separate subseries. When $n>(1+\epsilon)\tfrac{t}{\eta}-x$, there is no stationary point in the interval of integration and we can use integration by parts as before. When $n<(1-\epsilon)\tfrac{t}{\eta}-x$, we shall rewrite the integral along $i[\eta,\infty)$ as the difference of an integral along $i[0,\eta)$, which no longer contains a stationary point, and an integral along $i[0,\infty)$, which can be computed explicitly.

In analogy with equation (\ref{GL-summand-IbP}), repeatedly integrating by parts in the summand of (\ref{GU-series}) gives
\begin{multline}
\label{GU-summand-IbP1}
\int_{i\eta}^{i\infty}e^{-(x+n)z}z^{s-1}\,\mathrm{d}z=\sum_{j=0}^{N-1}e^{-(x+n)z+it\log z}\Big(\tfrac{1}{x+n-\frac{it}{z}}\cdot\tfrac{\mathrm{d}}{\mathrm{d}z}\Big)^j\Big(\tfrac{z^{\sigma-1}}{x+n-\frac{it}{z}}\Big)\Bigg|_{z=i\eta} \\ +\int_{i\eta}^{i\infty}e^{-(x+n)z+it\log z}D_N(z;x+n;\sigma,t)\,\mathrm{d}z.
\end{multline}
Similarly,
\begin{multline}
\label{GU-summand-IbP2}
\int_{0}^{i\eta}e^{-(x+n)z}z^{s-1}\,\mathrm{d}z=-\sum_{j=0}^{N-1}e^{-(x+n)z+it\log z}\Big(\tfrac{1}{x+n-\frac{it}{z}}\cdot\tfrac{\mathrm{d}}{\mathrm{d}z}\Big)^j\Big(\tfrac{z^{\sigma-1}}{x+n-\frac{it}{z}}\Big)\Bigg|_{z=i\eta} \\ +\int_{0}^{i\eta}e^{-(x+n)z+it\log z}D_N(z;x+n;\sigma,t)\,\mathrm{d}z,
\end{multline}
each of (\ref{GU-summand-IbP1}) and (\ref{GU-summand-IbP2}) being valid for any $N\in\mathbb{N}$. To derive each of these identities, we have used the fact that for every $j$, the summand in the $\sum_j$ series tends to zero as $|\eta|$ tends to either $0$ or $\infty$ with $\eta$ on the imaginary axis. This follows by approximating each part of the summand, e.g. by a power of $z$.

\begin{lem}
\label{GU-DN-estimates}
If $n<(1-\epsilon)\tfrac{t}{\eta}-x$ and $z\in i[0,\eta]$, then
\begin{equation}
\label{GU-DN-estimate1}
D_N(z;x+n;\sigma,t)=O\Big((2N-1)!!(N+1)^2|z|^{\sigma-1}t^{-N}\epsilon^{-2N}\Big).
\end{equation}
If $n>(1+\epsilon)\tfrac{t}{\eta}-x$ and $z\in i[\eta,\infty)$, then
\begin{equation}
\label{GU-DN-estimate2}
D_N(z;x+n;\sigma,t)=O\Big((2N-1)!!(N+1)^2|z|^{\sigma-N-1}(x+n)^{-N}\epsilon^{-2N}(1+\epsilon)^{2N}\Big).
\end{equation}
Both of these estimates are uniform in all parameters.
\end{lem}

\begin{proof}
The argument here is similar to the argument used in Lemma \ref{GL-DN-estimate}, starting from the expression (\ref{DN-series}) for $D_N$.

\textbf{Case 1:} $\boldsymbol{n<(1-\epsilon)\tfrac{t}{\eta}-x,z\in i[0,\eta]}.$

In this case, $|(x+n)z|\leq(x+n)\eta<(1-\epsilon)t$, and thus
\begin{align*}
D_N(z;x+n;\sigma,t)&=O\Bigg((2N-1)!!|z|^{\sigma-1}\sum_{b=0}^N\sum_{c=0}^N\Big|\tfrac{(x+n)z}{t}\Big|^{N-b}|\sigma|^c\Big|\tfrac{t}{((x+n)z-it)^2}\Big|^N\Bigg) \\
&=O\bigg((2N-1)!!(N+1)^2|z|^{\sigma-1}\Big|\tfrac{t}{((x+n)z-it)^2}\Big|^N\bigg).
\end{align*}
Since  $(x+n)z$ is positive imaginary with modulus at most $(1-\epsilon)t$, it follows that $\big|(x+n)z-it\big|>\epsilon t$, and thus
\begin{align*}
D_N&=O\Big((2N-1)!!(N+1)^2|z|^{\sigma-1}t^N(\epsilon t)^{-2N}\Big),
\end{align*}
as required.

\textbf{Case 2:} $\boldsymbol{n>(1+\epsilon)\tfrac{t}{\eta}-x,z\in i[\eta,\infty)}.$

In this case, $|(x+n)z|\geq(x+n)\eta>(1+\epsilon)t$, and thus
\begin{align*}
D_N(z;x+n;\sigma,t)&=O\Bigg((2N-1)!!|z|^{\sigma-1}\sum_{b=0}^N\sum_{c=0}^N\Big|\tfrac{t}{(x+n)z}\Big|^b|\sigma|^c\Big|\tfrac{(x+n)z}{((x+n)z-it)^2}\Big|^N\Bigg) \\
&=O\bigg((2N-1)!!(N+1)^2|z|^{\sigma-1}\Big|\tfrac{(x+n)z}{((x+n)z-it)^2}\Big|^N\bigg).
\end{align*}
Since $(x+n)z$ is positive imaginary with modulus at least $(1+\epsilon)t$, it follows that $\big|(x+n)z-it\big|>\tfrac{\epsilon}{1+\epsilon}\big|(x+n)z\big|$, and thus
\begin{align*}
D_N&=O\Big((2N-1)!!(N+1)^2|z|^{\sigma-1}\big|(x+n)z\big|^{-N}\big(\tfrac{\epsilon}{1+\epsilon}\big)^{-2N}\Big),
\end{align*}
which yields the desired estimate.

As in Lemma \ref{GL-DN-estimate}, all bounds are uniform and the $O$-constants can be taken to be $1$.
\end{proof}

\begin{lem}
\label{GU-intDN-estimates}
We have the following two estimates, uniform in $\sigma$, $t$, $\eta$, $x$, $\epsilon$, and $N\geq1$ satisfying Assumption \ref{Assumption}:
\begin{equation*}
\sum_{n=1}^{\big\lfloor\tfrac{t}{\eta}-x\big\rfloor}\int_{0}^{i\eta}e^{-(x+n)z+it\log z}D_N(z;x+n;\sigma,t)\,\mathrm{d}z=O\big((2N+1)!!(N+1)^2e^{-\pi t/2}\sigma^{-1}\eta^{\sigma-N-1}\epsilon^{-2N-2}\big),
\end{equation*}
and
\begin{equation*}
\sum_{n=\big\lceil\tfrac{t}{\eta}-x\big\rceil}^{\infty}\int_{i\eta}^{i\infty}e^{-(x+n)z+it\log z}D_N(z;x+n;\sigma,t)\,\mathrm{d}z=O\big((2N+1)!!(N+1)^2e^{-\pi t/2}\eta^{\sigma-N-1}\epsilon^{-2N-2}(1+\epsilon)^{2N+2}\big).
\end{equation*}
\end{lem}

\begin{proof}
We shall use the estimates from Lemma \ref{GU-DN-estimates}. Let $_UI_1$ and $_UI_2$ denote the two expressions we need to estimate. First, by (\ref{GU-DN-estimate1}) we find the following estimate:
\begin{align*}
_UI_1&=O\Bigg(\sum_{n=1}^{\big\lfloor\tfrac{t}{\eta}-x\big\rfloor}\int_{0}^{i\eta}e^{-(x+n)z+it\log z}(2N-1)!!(N+1)^2|z|^{\sigma-1}t^{-N}\epsilon^{-2N}\,\mathrm{d}z\Bigg) \\
&=O\Bigg((2N-1)!!(N+1)^2t^{-N}\epsilon^{-2N}\sum_{n=1}^{\big\lfloor\tfrac{t}{\eta}-x\big\rfloor}\int_{0}^{i\eta}e^{-\pi t/2}|z|^{\sigma-1}\,\mathrm{d}z\Bigg) \\
&=O\Bigg((2N-1)!!(N+1)^2e^{-\pi t/2}\big(\tfrac{\eta^{\sigma}}{\sigma}\big)t^{-N}\epsilon^{-2N}\sum_{n=1}^{\big\lfloor\tfrac{t}{\eta}-x\big\rfloor}1\Bigg) \\
&=O\big((2N-1)!!(N+1)^2e^{-\pi t/2}\sigma^{-1}\eta^{\sigma-1}t^{-N+1}\epsilon^{-2N}\big).
\end{align*}
We can assume $\eta<t$ (otherwise this expression is non-existent), so $t^{-N+1}<\eta^{-N+1}$ and therefore
\begin{equation}
\label{GU-I1-estimate}
_UI_1=O\big((2N-1)!!(N+1)^2e^{-\pi t/2}\sigma^{-1}\eta^{\sigma-N}\epsilon^{-2N}\big).
\end{equation}
Applying integration by parts once to the original expression for $_UI_1$ gives the expression
\begin{multline*}
_UI_1=\sum_{n=1}^{\big\lfloor\tfrac{t}{\eta}-x\big\rfloor}\Bigg(-\Big[e^{-(x+n)z+it\log z}\Big(\tfrac{1}{x+n-\tfrac{it}{z}}\Big)D_N(z;x+n;\sigma,t)\Big]_0^{i\eta} \\ +\int_{0}^{i\eta}e^{-(x+n)z+it\log z}D_{N+1}(z;x+n;\sigma,t)\,\mathrm{d}z\Bigg).
\end{multline*}
Using equation (\ref{GU-DN-estimate1}) again for the first half of this and equation (\ref{GU-I1-estimate}) (with $N$ replaced by $N+1$) for the second half, we find:
\begin{align*}
\begin{split}
_UI_1&=O\Bigg(\sum_{n=1}^{\big\lfloor\tfrac{t}{\eta}-x\big\rfloor}\Big[\tfrac{e^{-\pi t/2}z}{(x+n)z-it}(2N-1)!!(N+1)^2|z|^{\sigma-1}t^{-N}\epsilon^{-2N}\Big]_{z=i\eta}\Bigg) \\ &\hspace{7cm}+O\big((2N+1)!!(N+2)^2e^{-\pi t/2}\sigma^{-1}\eta^{\sigma-N-1}\epsilon^{-2N-2}\big)
\end{split} \\
&=O\Big(\tfrac{e^{-\pi t/2}\eta}{\epsilon t}(2N+1)!!(N+1)\eta^{\sigma-2}t^{-N+1}\epsilon^{-2N}\Big)+O\big((2N+1)!!(N+1)^2e^{-\pi t/2}\sigma^{-1}\eta^{\sigma-N-1}\epsilon^{-2N-2}\big) \\
&=O\Big((2N+1)!!e^{-\pi t/2}\big[(N+1)\eta^{\sigma-1}t^{-N}\epsilon^{-2N-1}+(N+1)^2\sigma^{-1}\eta^{\sigma-N-1}\epsilon^{-2N-2}\big]\Big) \\
&=O\big((2N+1)!!(N+1)^2e^{-\pi t/2}\sigma^{-1}\eta^{\sigma-N-1}\epsilon^{-2N-2}\big),
\end{align*}
where we have again used the estimates $(x+n)\eta-t>\epsilon t$ and $t^{-N}<\eta^{-N}$. Once again, all $O$-constants are uniform.

Second, by (\ref{GU-DN-estimate2}) we have:
\begin{align}
\nonumber _UI_2&=O\Bigg(\sum_{n=\big\lceil\tfrac{t}{\eta}-x\big\rceil}^{\infty}\int_{i\eta}^{i\infty}e^{-(x+n)z+it\log z}(2N-1)!!(N+1)^2|z|^{\sigma-N-1}(x+n)^{-N}\epsilon^{-2N}(1+\epsilon)^{2N}\,\mathrm{d}z\Bigg) \\
\nonumber &=O\Bigg((2N-1)!!(N+1)^2\epsilon^{-2N}(1+\epsilon)^{2N}\sum_{n=\big\lceil\tfrac{t}{\eta}-x\big\rceil}^{\infty}\int_{i\eta}^{i\infty}e^{-\pi t/2}|z|^{\sigma-N-1}(x+n)^{-N}\,\mathrm{d}z\Bigg) \\
\nonumber &=O\Bigg((2N-1)!!(N+1)^2e^{-\pi t/2}(N-\sigma)^{-1}\eta^{\sigma-N}\epsilon^{-2N}(1+\epsilon)^{2N}\sum_{n=\big\lceil\tfrac{t}{\eta}-x\big\rceil}^{\infty}(x+n)^{-N}\Bigg) \\
\label{GU-I2-estimate} &=O\big((2N-1)!!(N+1)^2e^{-\pi t/2}\eta^{\sigma-N}\epsilon^{-2N}(1+\epsilon)^{2N}\big),
\end{align}
provided that $N\geq2$.

Applying integration by parts once to the original expression for $_UI_2$ gives the expression
\begin{multline*}
_UI_2=\sum_{n=\big\lceil\tfrac{t}{\eta}-x\big\rceil}^{\infty}\Bigg(-\Big[e^{-(x+n)z+it\log z}\Big(\tfrac{1}{x+n-\tfrac{it}{z}}\Big)D_N(z;x+n;\sigma,t)\Big]_{i\eta}^{i\infty} \\ +\int_{i\eta}^{i\infty}e^{-(x+n)z+it\log z}D_{N+1}(z;x+n;\sigma,t)\,\mathrm{d}z\Bigg).
\end{multline*}
Using equation (\ref{GU-DN-estimate2}) again for the first half of this expression and equation (\ref{GU-I2-estimate}) (with $N$ replaced by $N+1$, so that our $N\geq2$ assumption becomes only $N\geq1$) for the second half, we find:
\begin{align*}
\begin{split}
_UI_2&=O\Bigg(\sum_{n=\big\lceil\tfrac{t}{\eta}-x\big\rceil}^{\infty}\Big[\tfrac{e^{-\pi t/2}z}{(x+n)z-it}(2N-1)!!(N+1)^2|z|^{\sigma-N-1}(x+n)^{-N}\epsilon^{-2N}(1+\epsilon)^{2N}\Big]_{z=i\eta}\Bigg) \\ &\hspace{5cm}+O\big((2N+1)!!(N+2)^2e^{-\pi t/2}\eta^{\sigma-N-1}\epsilon^{-2N-2}(1+\epsilon)^{2N+2}\big)
\end{split} \\
\begin{split}
&=O\Bigg(\sum_{n=\big\lceil\tfrac{t}{\eta}-x\big\rceil}^{\infty}\tfrac{e^{-\pi t/2}\eta}{\tfrac{\epsilon}{1+\epsilon}(x+n)\eta}(2N+1)!!(N+1)\eta^{\sigma-N-1}(x+n)^{-N}\epsilon^{-2N}(1+\epsilon)^{2N}\Bigg) \\ &\hspace{5cm}+O\big((2N+1)!!(N+1)^2e^{-\pi t/2}\eta^{\sigma-N-1}\epsilon^{-2N-2}(1+\epsilon)^{2N+2}\big)
\end{split} \\
&=O\Big(e^{-\pi t/2}(2N+1)!!\big[(N+1)\eta^{\sigma-N-1}\epsilon^{-2N-1}(1+\epsilon)^{2N+1}+(N+1)^2\eta^{\sigma-N-1}\epsilon^{-2N-2}(1+\epsilon)^{2N+2}\big]\Big) \\
&=O\big((2N+1)!!(N+1)^2e^{-\pi t/2}\eta^{\sigma-N-1}\epsilon^{-2N-2}(1+\epsilon)^{2N+2}\big),
\end{align*}
where we have used again the estimate $(x+n)\eta-t>\tfrac{\epsilon}{1+\epsilon}(x+n)\eta$. Once again, all $O$-constants are uniform.
\end{proof}

\begin{lem}
\label{GU-estimate}
We have the following estimate for $G_U$, uniform in $\sigma$, $t$, $\eta$, $x$, $\epsilon$, and $N\geq1$ satisfying Assumption \ref{Assumption}:
\begin{multline*}
G_U(\sigma,t;\eta;x)=\frac{e^{-i\pi s/2}}{(2\pi)^s}\sum_{n=1}^{\big\lfloor\tfrac{t}{\eta}-x\big\rfloor}\frac{\Gamma(s)}{(x+n)^s}+\frac{e^{-i\pi\sigma/2}e^{it\log\eta}}{(2\pi)^s}\sum_{n=1}^{M}\sum_{j=0}^{N-1}e^{-i(x+n)\eta}\bigg[\Big(\tfrac{1}{x+n-\frac{it}{z}}\cdot\tfrac{\mathrm{d}}{\mathrm{d}z}\Big)^j\Big(\tfrac{z^{\sigma-1}}{x+n-\frac{it}{z}}\Big)\bigg]_{z=i\eta} \\ +O\big((2N+1)!!(N+1)^2\sigma^{-1}\eta^{\sigma-N-1}\epsilon^{-2N-2}(1+\epsilon)^{2N+2}\big),
\end{multline*}
where $M$ is a finite number depending only on $N$ and $\eta$.
\end{lem}

\begin{proof}
By (\ref{GU-series}) with $\phi_1=\tfrac{\pi}{2}$, we find:
\begin{align*}
G_U(\sigma,t;\eta;x)&=\tfrac{e^{-i\pi s/2}}{(2\pi)^s}\sum_{n=1}^{\infty}\int_{i\eta}^{i\infty}e^{-(x+n)z}z^{s-1}\,\mathrm{d}z \\
\begin{split}
&=\frac{e^{-i\pi s/2}}{(2\pi)^s}\Bigg[\sum_{n=1}^{\big\lfloor\tfrac{t}{\eta}-x\big\rfloor}\bigg(\int_{0}^{i\infty}e^{-(x+n)z}z^{s-1}\,\mathrm{d}z-\int_{0}^{i\eta}e^{-(x+n)z}z^{s-1}\,\mathrm{d}z\bigg) \\ &\hspace{8cm}+\sum_{n=\big\lceil\tfrac{t}{\eta}-x\big\rceil}^{\infty}\int_{i\eta}^{i\infty}e^{-(x+n)z}z^{s-1}\,\mathrm{d}z\Bigg]
\end{split} \\
&=\frac{e^{-i\pi s/2}}{(2\pi)^s}\Bigg[\sum_{n=1}^{\big\lfloor\tfrac{t}{\eta}-x\big\rfloor}\bigg(\frac{\Gamma(s)}{(x+n)^s}-\int_{0}^{i\eta}e^{-(x+n)z}z^{s-1}\,\mathrm{d}z\bigg)+\sum_{n=\big\lceil\tfrac{t}{\eta}-x\big\rceil}^{\infty}\int_{i\eta}^{i\infty}e^{-(x+n)z}z^{s-1}\,\mathrm{d}z\Bigg].
\end{align*}
Substituting (\ref{GU-summand-IbP1}) and (\ref{GU-summand-IbP2}) into the above expression yields:
\begin{multline*}
G_U=\frac{e^{-i\pi s/2}}{(2\pi)^s}\Bigg[\sum_{n=1}^{\big\lfloor\tfrac{t}{\eta}-x\big\rfloor}\frac{\Gamma(s)}{(x+n)^s}-\sum_{n=1}^{\big\lfloor\tfrac{t}{\eta}-x\big\rfloor}\bigg(-\sum_{j=0}^{N-1}e^{-(x+n)z+it\log z}\Big(\tfrac{1}{x+n-\frac{it}{z}}\cdot\tfrac{\mathrm{d}}{\mathrm{d}z}\Big)^j\Big(\tfrac{z^{\sigma-1}}{x+n-\frac{it}{z}}\Big)\Bigg|_{z=i\eta} \\ \hspace{7cm}+\int_{0}^{i\eta}e^{-(x+n)z+it\log z}D_N(z;x+n;\sigma,t)\,\mathrm{d}z\bigg) \\ +\sum_{n=\big\lceil\tfrac{t}{\eta}-x\big\rceil}^{\infty}\bigg(\sum_{j=0}^{N-1}e^{-(x+n)z+it\log z}\Big(\tfrac{1}{x+n-\frac{it}{z}}\cdot\tfrac{\mathrm{d}}{\mathrm{d}z}\Big)^j\Big(\tfrac{z^{\sigma-1}}{x+n-\frac{it}{z}}\Big)\Bigg|_{z=i\eta} \\ +\int_{i\eta}^{i\infty}e^{-(x+n)z+it\log z}D_N(z;x+n;\sigma,t)\,\mathrm{d}z\bigg)\Bigg],
\end{multline*}
where we have used Jordan's lemma and the fact that $x+n$ is positive to obtain \[\int_{0}^{i\infty}e^{-(x+n)z}z^{s-1}\,\mathrm{d}z=(n+x)^{-s}\Gamma(s).\] Now substituting the results of Lemma \ref{GU-intDN-estimates} into the expression for $G_U$ yields
\begin{multline*}
G_U=\frac{e^{-i\pi s/2}}{(2\pi)^s}\Bigg[\sum_{n=1}^{\big\lfloor\tfrac{t}{\eta}-x\big\rfloor}\frac{\Gamma(s)}{(x+n)^s}+\sum_{n=1}^{\infty}\sum_{j=0}^{N-1}e^{-(x+n)z+it\log z}\Big(\tfrac{1}{x+n-\frac{it}{z}}\cdot\tfrac{\mathrm{d}}{\mathrm{d}z}\Big)^j\Big(\tfrac{z^{\sigma-1}}{x+n-\frac{it}{z}}\Big)\Bigg|_{z=i\eta} \\ +O\big((2N+1)!!(N+1)^2e^{-\pi t/2}\sigma^{-1}\eta^{\sigma-N-1}\epsilon^{-2N-2}(1+\epsilon)^{2N+2}\big)\Bigg].
\end{multline*}
Finally, using $e^{-i\pi s/2}=e^{-i\pi\sigma/2}e^{\pi t/2}$ and $e^{it\log(i\eta)}=e^{it\log\eta}e^{-\pi t/2}$ gives
\begin{multline*}
G_U=\frac{e^{-i\pi s/2}}{(2\pi)^s}\sum_{n=1}^{\big\lfloor\tfrac{t}{\eta}-x\big\rfloor}\frac{\Gamma(s)}{(x+n)^s}+\frac{e^{-i\pi\sigma/2}e^{it\log\eta}}{(2\pi)^s}\sum_{n=1}^{\infty}\sum_{j=0}^{N-1}e^{-i(x+n)\eta}\bigg[\Big(\tfrac{1}{x+n-\frac{it}{z}}\cdot\tfrac{\mathrm{d}}{\mathrm{d}z}\Big)^j\Big(\tfrac{z^{\sigma-1}}{x+n-\frac{it}{z}}\Big)\bigg]_{z=i\eta} \\ +O\big((2N+1)!!(N+1)^2\sigma^{-1}\eta^{\sigma-N-1}\epsilon^{-2N-2}(1+\epsilon)^{2N+2}\big).
\end{multline*}
The uniformity of the $O$-bound is inherited from Lemma \ref{GU-intDN-estimates}. In order to derive the final result, we just need to find $M(N,\eta)$ large enough so that
\begin{multline}
\label{GU-M-tail}
\sum_{n=M+1}^{\infty}\sum_{j=0}^{N-1}e^{-i(x+n)\eta}\bigg[\Big(\tfrac{1}{x+n-\frac{it}{z}}\cdot\tfrac{\mathrm{d}}{\mathrm{d}z}\Big)^j\Big(\tfrac{z^{\sigma-1}}{x+n-\frac{it}{z}}\Big)\bigg]_{z=i\eta} \\ =O\big((2N+1)!!(N+1)^2\sigma^{-1}\eta^{\sigma-N-1}\epsilon^{-2N-2}(1+\epsilon)^{2N+2}\big).
\end{multline}
Using the definition of $D_N$ and the bounds (\ref{GU-DN-estimate1}) and (\ref{GU-DN-estimate2}) for $D_N$, we can estimate the left hand side of (\ref{GU-M-tail}) as follows:
\begin{align*}
&{\color{white}=}\sum_{n=M+1}^{\infty}\sum_{j=0}^{N-1}\frac{e^{-i(x+n)\eta}}{x+n-\frac{t}{\eta}}D_j(i\eta;x+n;\sigma,t) \\
\begin{split}
&=\sum_{n=M+1}^{\infty}\frac{e^{-i(x+n)\eta}(i\eta)^{\sigma-1}}{x+n-\frac{t}{\eta}} \\
&\hspace{1.5cm}+\sum_{n=M+1}^{\infty}\sum_{j=1}^{N-1}O\bigg(\frac{1}{x+n-\frac{t}{\eta}}(2j-1)!!(j+1)^2\eta^{\sigma-1}\max\big(t,(x+n)\eta\big)^{-j}\epsilon^{-2j}(1+\epsilon)^{2j}\bigg)
\end{split} \\
\begin{split}
&=O\bigg(\eta^{\sigma-1}\sum_{n=M+1}^{\infty}\frac{e^{in\eta}}{x+n-\frac{t}{\eta}}\bigg) \\
&\hspace{1.5cm}+\sum_{j=1}^{N-1}O\bigg((2j-1)!!(j+1)^2\eta^{\sigma-1}\epsilon^{-2j}(1+\epsilon)^{2j}\sum_{n=M+1}^{\infty}\frac{1}{x+n-\frac{t}{\eta}}\big((x+n)\eta\big)^{-j}\bigg).
\end{split}
\end{align*}
In both Case 1 and Case 2 of Lemma \ref{GU-DN-estimates}, we have $\big|(x+n)\eta-t\big|>\frac{\epsilon}{1+\epsilon}(x+n)\eta$. So the left hand side of (\ref{GU-M-tail}) is:
\begin{multline*}
O\bigg(\eta^{\sigma-1}\sum_{n=M+1}^{\infty}\frac{e^{in\eta}}{\frac{\epsilon}{1+\epsilon}(x+n)}\bigg) \\ +O\bigg((2N-3)!!N^2\epsilon^{-2N+2}(1+\epsilon)^{2N-2}\sum_{j=1}^{N-1}\eta^{\sigma-j-1}\sum_{n=M+1}^{\infty}\frac{1+\epsilon}{\epsilon}(x+n)^{-j-1}\bigg).
\end{multline*}
All of the infinite series in this expression are convergent, so we can simply choose $M$ large enough so that \[\sum_{n=M+1}^{\infty}\frac{e^{in\eta}}{x+n}\leq\eta^{-N}\] and \[\sum_{n=M+1}^{\infty}(x+n)^{-j-1}\leq\eta^{j-N}\] for $j=1,2,\dots,N-1$. These are exactly the same conditions as we found in the proof of Lemma \ref{GL-estimate}. So for any $M$ satisfying (\ref{GL-M-condition1}) and (\ref{GL-M-condition2}), we have the required bound (\ref{GU-M-tail}), and the result holds.
\end{proof}

\section{Asymptotics for $G_B$}
The semicircle $\hat{C}^0_{\eta}$ is in the left half plane, and the series $\sum_{n=0}^{\infty}e^{nz}=\frac{1}{1-e^{z}}$ is locally uniformly convergent for $\mathrm{Re}(z)<0$. So we can interchange the series and integral and then use Cauchy's theorem:
\begin{align*}
G_B(\sigma,t;\eta;x)&=\tfrac{e^{-i\pi s/2}}{(2\pi)^s}\sum_{n=0}^{\infty}\int_{\hat{C}^0_{\eta}}\big(e^{(1+x)z}-e^{-xz}\big)e^{nz}z^{s-1}\,\mathrm{d}z \\
&=\tfrac{e^{-i\pi s/2}}{(2\pi)^s}\sum_{n=0}^{\infty}\int_{i\eta}^{0}\big(e^{(n+1+x)z}-e^{(n-x)z}\big)z^{s-1}\,\mathrm{d}z.
\end{align*}
Substituting $\tilde{z}=e^{-i\pi}z$ and then changing notation back to $z$, we find
\begin{align}
\nonumber G_B&=\tfrac{e^{i\pi s/2}}{(2\pi)^s}\sum_{n=0}^{\infty}\int_{-i\eta}^{0}\big(e^{-(n+1+x)\tilde{z}}-e^{-(n-x)\tilde{z}}\big)\tilde{z}^{s-1}\,\mathrm{d}\tilde{z} \\
\label{GB-series} &=\tfrac{e^{i\pi s/2}}{(2\pi)^s}\sum_{n=0}^{\infty}\bigg(\int_{-i\eta}^{0}e^{-(n+1+x)z}z^{s-1}\,\mathrm{d}z-\int_{-i\eta}^{0}e^{-(n-x)z}z^{s-1}\,\mathrm{d}z\bigg).
\end{align}
The first integrand is $e^{-(n+1+x)z+it\log z}z^{\sigma-1}$, which has a stationary point at $z=\tfrac{it}{n+1+x}$, and this value of $z$ is not in the interval of integration since $\eta>0$. The second integrand is $e^{-(n-x)z+it\log z}z^{\sigma-1}$; this has a stationary point at $z=\tfrac{it}{n-x}$, which is within the interval of integration iff \[\tfrac{it}{n-x}\in -i[0,\eta], \text{ i.e. } 0\leq\tfrac{t}{x-n}\leq\eta, \text{ i.e. } n\leq x-\tfrac{t}{\eta}.\] Hence, in analogy with the asymptotics of $G_U$, we shall split the sum over $n$. In this case, the situation is slightly more complicated, because we also need to consider different cases according to whether $n-x$ is positive or negative. We therefore have three different cases:
\begin{align*}
&n<x-(1+\epsilon)\tfrac{t}{\eta}; \\
x-(1-\epsilon)\tfrac{t}{\eta}<\;&n\leq x; \\
x<\;&n.
\end{align*}
The possibility of $x-(1+\epsilon)\tfrac{t}{\eta}\leq n\leq x-(1-\epsilon)\tfrac{t}{\eta}$ is ruled out by Assumption \ref{Assumption}.

We cannot consider the two sums \[\sum_{n=0}^{\infty}\int_{-i\eta}^{0}e^{-(n+1+x)z}z^{s-1}\,\mathrm{d}z\;\;,\;\;\sum_{n=0}^{\infty}\int_{-i\eta}^{0}e^{-(n-x)z}z^{s-1}\,\mathrm{d}z\] independently, since each of these series diverges on its own. Thus, for the case $n>x$, which is the only one of the three cases to permit infinitely many values of $n$, we need to analyse both of these series together. In this case, the first step involves substituting $w=(n+1+x)z$ and $w=(n-x)z$ respectively into the two integrals in the summand:
\begin{align*}
&\int_{-i\eta}^{0}e^{-(n+1+x)z}z^{s-1}\,\mathrm{d}z-\int_{-i\eta}^{0}e^{-(n-x)z}z^{s-1}\,\mathrm{d}z \\
=(n+1+x)^{-s}&\int_{-i\eta(n+1+x)}^{0}e^{-w}w^{s-1}\,\mathrm{d}w-(n-x)^{-s}\int_{-i\eta(n-x)}^{0}e^{-w}w^{s-1}\,\mathrm{d}w \\
=(n+1+x)^{-s}&\int_{-i\eta(n+1+x)}^{-i\eta(n-x)}e^{-w}w^{s-1}\,\mathrm{d}w+\Big((n+1+x)^{-s}-(n-x)^{-s}\Big)\int_{-i\eta(n-x)}^{0}e^{-w}w^{s-1}\,\mathrm{d}w.
\end{align*}
For the case $n<x-(1+\epsilon)\tfrac{t}{\eta}$, we proceed in the same way as with $G_U$: we rewrite the integral $\int_{-i\eta}^0$, which contains a stationary point, as the difference of the integral $\int_{-\infty}^{-i\eta}$, which does not contain a stationary point, and the integral $\int_{-\infty}^0$, which can be computed explicitly.

Substituting the above into (\ref{GB-series}), we find:
\begin{multline}
\label{GB-prelemma}
G_B=\frac{e^{i\pi s/2}}{(2\pi)^s}\sum_{n=0}^{\big\lfloor x-\tfrac{t}{\eta}\big\rfloor}\bigg[\int_{-i\eta}^{0}e^{-(n+1+x)z}z^{s-1}\,\mathrm{d}z+\int_{-i\infty}^{-i\eta}e^{-(n-x)z}z^{s-1}\,\mathrm{d}z-\int_{-i\infty}^{0}e^{-(n-x)z}z^{s-1}\,\mathrm{d}z\bigg] \\
+\frac{e^{i\pi s/2}}{(2\pi)^s}\sum_{n=\big\lceil x-\tfrac{t}{\eta}\big\rceil}^{\lfloor x \rfloor}\bigg[\int_{-i\eta}^{0}e^{-(n+1+x)z}z^{s-1}\,\mathrm{d}z+\int_{0}^{-i\eta}e^{-(n-x)z}z^{s-1}\,\mathrm{d}z\bigg] \\
\hspace{-2cm}+\frac{e^{i\pi s/2}}{(2\pi)^s}\sum_{n=\lfloor x\rfloor+1}^{\infty}\bigg[(n+1+x)^{-s}\int_{-i\eta(n+1+x)}^{-i\eta(n-x)}e^{-w}w^{s-1}\,\mathrm{d}w \\ +\Big((n+1+x)^{-s}-(n-x)^{-s}\Big)\int_{-i\eta(n-x)}^{0}e^{-w}w^{s-1}\,\mathrm{d}w\bigg].
\end{multline}

The explicit term in this sum is given by
\begin{equation}
\label{GB-explicit}
-\int_{-i\infty}^{0}e^{(x-n)z}z^{s-1}\,\mathrm{d}z=e^{-i\pi s}(x-n)^{-s}\Gamma(s),
\end{equation}
where we have used Jordan's lemma and the assumption that $x-n$ is positive.

For the other terms, we use again repeated integration by parts:
\begin{align}
\begin{split}
\label{GB-summand-IbP1}
&\int_{-i\eta}^{0}e^{-(x+n+1)z}z^{s-1}\,\mathrm{d}z=\sum_{j=0}^{N-1}e^{-(x+n+1)z+it\log z}\Big(\tfrac{1}{x+n+1-\frac{it}{z}}\cdot\tfrac{\mathrm{d}}{\mathrm{d}z}\Big)^j\Big(\tfrac{z^{\sigma-1}}{x+n+1-\frac{it}{z}}\Big)\Bigg|_{z=-i\eta} \\ &\hspace{6cm}+\int_{-i\eta}^{0}e^{-(x+n+1)z+it\log z}D_N(z;x+n+1;\sigma,t)\,\mathrm{d}z;
\end{split} \\
\begin{split}
\label{GB-summand-IbP2}
&\int_{-i\infty}^{-i\eta}e^{(x-n)z}z^{s-1}\,\mathrm{d}z=-\sum_{j=0}^{N-1}e^{(x-n)z+it\log z}\Big(\tfrac{1}{n-x-\frac{it}{z}}\cdot\tfrac{\mathrm{d}}{\mathrm{d}z}\Big)^j\Big(\tfrac{z^{\sigma-1}}{n-x-\frac{it}{z}}\Big)\Bigg|_{z=-i\eta} \\ &\hspace{7cm}+\int_{-i\infty}^{-i\eta}e^{(x-n)z+it\log z}D_N(z;n-x;\sigma,t)\,\mathrm{d}z;
\end{split} \\
\begin{split}
\label{GB-summand-IbP3}
&\int_{0}^{-i\eta}e^{(x-n)z}z^{s-1}\,\mathrm{d}z=-\sum_{j=0}^{N-1}e^{(x-n)z+it\log z}\Big(\tfrac{1}{n-x-\frac{it}{z}}\cdot\tfrac{\mathrm{d}}{\mathrm{d}z}\Big)^j\Big(\tfrac{z^{\sigma-1}}{n-x-\frac{it}{z}}\Big)\Bigg|_{z=-i\eta} \\ &\hspace{7cm}+\int_{0}^{-i\eta}e^{(x-n)z+it\log z}D_N(z;n-x;\sigma,t)\,\mathrm{d}z;
\end{split} \\
\begin{split}
\label{GB-summand-IbP4}
&\int_{-i\eta(n+1+x)}^{-i\eta(n-x)}e^{-w}w^{s-1}\,\mathrm{d}w=\sum_{j=0}^{N-1}\bigg[e^{-w+it\log w}\Big(\tfrac{1}{1-\frac{it}{w}}\cdot\tfrac{\mathrm{d}}{\mathrm{d}w}\Big)^j\Big(\tfrac{w^{\sigma-1}}{1-\frac{it}{w}}\Big)\bigg]^{w=-i\eta(n+1+x)}_{w=-i\eta(n-x)} \\ &\hspace{7cm}+\int_{-i\eta(n+1+x)}^{-i\eta(n-x)}e^{-w+it\log w}D_N(w;1;\sigma,t)\,\mathrm{d}w;
\end{split} \\
\begin{split}
\label{GB-summand-IbP5}
&\int_{-i\eta(n-x)}^{0}e^{-w}w^{s-1}\,\mathrm{d}w=\sum_{j=0}^{N-1}e^{-w+it\log w}\Big(\tfrac{1}{1-\frac{it}{w}}\cdot\tfrac{\mathrm{d}}{\mathrm{d}w}\Big)^j\Big(\tfrac{w^{\sigma-1}}{1-\frac{it}{w}}\Big)\Bigg|_{w=-i\eta(n-x)} \\ &\hspace{8cm}+\int_{-i\eta(n-x)}^{0}e^{-w+it\log w}D_N(w;1;\sigma,t)\,\mathrm{d}w.
\end{split}
\end{align}

\begin{lem}
\label{GB-intDN-estimate1}
For $n\leq x$, the remainder term in (\ref{GB-summand-IbP1}) can be uniformly approximated by \[O\big((2N+1)!!(N+1)^2e^{\pi t/2}\sigma^{-1}t^{\sigma-N-1}(x+n+1)^{-\sigma}\big).\]
\end{lem}

\begin{proof}
Let $_BI_1$ denote the expression we need to estimate, i.e.
\begin{equation*}
_BI_1\coloneqq\int_{-i\eta}^{0}e^{-(x+n+1)z+it\log z}D_N(z;x+n+1;\sigma,t)\,\mathrm{d}z.
\end{equation*}

By Lemma \ref{GL-DN-estimate}, we have \[D_N(z;x+n+1;\sigma,t)=O\Big((2N-1)!!(N+1)^2|z|^{\sigma-1}\max\big((x+n+1)|z|,t\big)^{-N}\Big),\] and therefore
\begin{align*}
_BI_1&=O\bigg(\int_{-i\eta}^{0}e^{-(x+n+1)z+it\log z}(2N-1)!!(N+1)^2|z|^{\sigma-1}\max\big((x+n+1)|z|,t\big)^{-N}\,\mathrm{d}z\bigg) \\
&=O\bigg((2N-1)!!(N+1)^2e^{\pi t/2}\int_{-i\eta}^{0}|z|^{\sigma-1}\max((x+n+1)|z|,t\big)^{-N}\,\mathrm{d}z\bigg).
\end{align*}
If $(x+n+1)\eta\leq t$, the above yields
\begin{align*}
_BI_1&=O\bigg((2N-1)!!(N+1)^2e^{\pi t/2}\int_{-i\eta}^{0}|z|^{\sigma-1}t^{-N}\,\mathrm{d}z\bigg) \\
&=O\Big((2N-1)!!(N+1)^2e^{\pi t/2}\big(\tfrac{\eta^{\sigma}}{\sigma}\big)t^{-N}\Big)=O\Big((2N-1)!!(N+1)^2e^{\pi t/2}\sigma^{-1}t^{\sigma-N}(x+n+1)^{-\sigma}\Big).
\end{align*}
If $(x+n+1)\eta>t$, we find
\begin{align*}
_BI_1&=O\Bigg((2N-1)!!(N+1)^2e^{\pi t/2}\bigg(\int_{\eta}^{t/(x+n+1)}(x+n+1)^{-N}z^{\sigma-N-1}\,\mathrm{d}z+\int_{t/(x+n+1)}^0z^{\sigma-1}t^{-N}\,\mathrm{d}z\bigg)\Bigg) \\
&=O\Bigg((2N-1)!!(N+1)^2e^{\pi t/2}\bigg[\tfrac{(x+n+1)^{-N}}{\sigma-N}\Big(\eta^{\sigma-N}-\big(\tfrac{t}{x+n+1}\big)^{\sigma-N}\Big)+\tfrac{t^{-N}}{\sigma}\big(\tfrac{t}{x+n+1}\big)^{\sigma}\bigg]\Bigg) \\
&=O\Bigg((2N-1)!!(N+1)^2e^{\pi t/2}\bigg[\frac{(x+n+1)^{\sigma-N}\eta^{\sigma-N}}{(N-\sigma)(x+n+1)^{\sigma}}+\frac{N}{N-\sigma}\cdot\frac{t^{\sigma-N}}{\sigma(x+n+1)^{\sigma}}\bigg]\Bigg) \\
&=O\Big((2N-1)!!(N+1)^2e^{\pi t/2}\sigma^{-1}t^{\sigma-N}(x+n+1)^{-\sigma}\Big),
\end{align*}
where the $O$-bound is uniform provided $N\geq2$. Thus, in both cases,
\begin{equation}
\label{GB-I1-estimate}
_BI_1=O\Big((2N-1)!!(N+1)^2e^{\pi t/2}\sigma^{-1}t^{\sigma-N}(x+n+1)^{-\sigma}\Big).
\end{equation}

Applying integration by parts once to the original expression for $_BI_1$ gives
\begin{multline*}
_BI_1=-\Big[e^{-(x+n+1)z+it\log z}\Big(\tfrac{1}{x+n+1-\tfrac{it}{z}}\Big)D_N(z;x+n+1;\sigma,t)\Big]_{-i\eta}^0 \\ +\int_{-i\eta}^0e^{-(x+n+1)z+it\log z}D_{N+1}(z;x+n+1;\sigma,t)\,\mathrm{d}z.
\end{multline*}
Using the results of Lemma \ref{GL-DN-estimate} for the first half of this expression, and equation (\ref{GB-I1-estimate}) (with $N$ replaced by $N+1$, so that our $N\geq2$ assumption becomes only $N\geq1$) for the second half, we find:
\begin{align*}
\begin{split}
_BI_1&=O\bigg(\tfrac{e^{\pi t/2}z}{(x+n+1)z-it}(2N-1)!!(N+1)^2|z|^{\sigma-1}\max\big((x+n+1)|z|,t\big)^{-N}\bigg|_{z=-i\eta}\bigg) \\ &\hspace{5cm}+O\Big((2N+1)!!(N+2)^2e^{\pi t/2}\sigma^{-1}t^{\sigma-N-1}(x+n+1)^{-\sigma}\Big)
\end{split} \\
\begin{split}
&=O\Big((2N+1)!!(N+1)e^{\pi t/2}\eta^{\sigma}\Big(\tfrac{\max((x+n+1)\eta,t)^{-N}}{(x+n+1)\eta+t}\Big)\Big) \\ &\hspace{5cm}+O\Big((2N+1)!!(N+1)^2e^{\pi t/2}\sigma^{-1}t^{\sigma-N-1}(x+n+1)^{-\sigma}\Big).
\end{split}
\end{align*}
If $(x+n+1)\eta\leq t$, the above yields
\begin{align*}
_BI_1&=O\Big((2N+1)!!(N+1)e^{\pi t/2}\Big[\eta^{\sigma}\cdot\tfrac{t^{-N}}{t}+(N+1)\sigma^{-1}t^{\sigma-N-1}(x+n+1)^{-\sigma}\Big]\Big) \\
&=O\Big((2N+1)!!(N+1)e^{\pi t/2}\Big[\big(\tfrac{t}{x+n+1}\big)^{\sigma}t^{-N-1}+(N+1)\sigma^{-1}t^{\sigma-N-1}(x+n+1)^{-\sigma}\Big]\Big) \\
&=O\big((2N+1)!!(N+1)^2e^{\pi t/2}\sigma^{-1}t^{\sigma-N-1}(x+n+1)^{-\sigma}\big).
\end{align*}
If $(x+n+1)\eta>t$, we find
\begin{align*}
_BI_1&=O\Big((2N+1)!!(N+1)e^{\pi t/2}\Big[\eta^{\sigma}\cdot\tfrac{(x+n+1)^{-N}\eta^{-N}}{(x+n+1)\eta}+(N+1)\sigma^{-1}t^{\sigma-N-1}(x+n+1)^{-\sigma}\Big]\Big) \\
&=O\Big((2N+1)!!(N+1)e^{\pi t/2}\Big[\big(\tfrac{t}{x+n+1}\big)^{\sigma-N-1}(x+n+1)^{-N-1}+(N+1)\sigma^{-1}t^{\sigma-N-1}(x+n+1)^{-\sigma}\Big]\Big) \\
&=O\big((2N+1)!!(N+1)^2e^{\pi t/2}\sigma^{-1}t^{\sigma-N-1}(x+n+1)^{-\sigma}\big).
\end{align*}
Thus, in both cases, we have the required estimate for $_BI_1$, and the $O$-bound is uniform in all parameters.
\end{proof}

\begin{lem}
\label{GB-intDN-estimate2}
For $n<x-(1+\epsilon)\tfrac{t}{\eta}$, the remainder term in (\ref{GB-summand-IbP2}) can be uniformly approximated by \[O\big((2N+1)!!(N+1)e^{\pi t/2}\eta^{\sigma-N-1}(x-n)^{-N-1}\epsilon^{-2N-2}(1+\epsilon)^{2N+2}\big).\] For $x-(1-\epsilon)\tfrac{t}{\eta}<n\leq x$, the remainder term in (\ref{GB-summand-IbP3}) can be uniformly approximated by \[O\big((2N+1)!!(N+1)^2e^{\pi t/2}\sigma^{-1}\eta^{\sigma}t^{-N-1}\epsilon^{-2N-2}\big).\]
\end{lem}

\begin{proof}
Let $_BI_2$ and $_BI_3$ denote the two expressions we need to estimate, i.e.
\begin{equation*}
_BI_2\coloneqq\int_{-i\infty}^{-i\eta}e^{(x-n)z+it\log z}D_N(z;n-x;\sigma,t)\,\mathrm{d}z
\end{equation*}
and
\begin{equation*}
_BI_3\coloneqq\int_{0}^{-i\eta}e^{(x-n)z+it\log z}D_N(z;n-x;\sigma,t)\,\mathrm{d}z.
\end{equation*}

For $_BI_2$, we can use the same argument as in Case 2 of Lemma \ref{GU-DN-estimates} to estimate $D_N$. Starting from the expression (\ref{DN-series}) and using the fact that $|(n-x)z|\geq(x-n)\eta>(1+\epsilon)t$, we find
\begin{align*}
D_N(z;n-x;\sigma,t)&=O\bigg((2N-1)!!|z|^{\sigma-1}\sum_{b=0}^N\sum_{c=0}^N\Big|\tfrac{t}{(n-x)z}\Big|^b|\sigma|^c\Big|\tfrac{(n-x)z}{((n-x)z-it)^2}\Big|^N\bigg) \\
&=O\bigg((2N-1)!!(N+1)^2|z|^{\sigma-1}\Big|\tfrac{(n-x)z}{((n-x)z-it)^2}\Big|^N\bigg).
\end{align*}
Since $(n-x)z$ is positive imaginary with modulus at least $(1+\epsilon)t$, it follows that \[\big|(n-x)z-it\big|>\tfrac{\epsilon}{1+\epsilon}\big|(n-x)z\big|,\] and thus
\begin{align*}
D_N&=O\Big((2N-1)!!(N+1)^2|z|^{\sigma-1}\big|(n-x)z\big|^{-N}\big(\tfrac{\epsilon}{1+\epsilon}\big)^{-2N}\Big) \\
&=O\big((2N-1)!!(N+1)^2|z|^{\sigma-N-1}(x-n)^{-N}\epsilon^{-2N}(1+\epsilon)^{2N}\big).
\end{align*}
Therefore,
\begin{align*}
_BI_2&=O\bigg(\int_{-i\infty}^{-i\eta}e^{(x-n)z+it\log z}(2N-1)!!(N+1)^2|z|^{\sigma-N-1}(x-n)^{-N}\epsilon^{-2N}(1+\epsilon)^{2N}\,\mathrm{d}z\bigg) \\
&=O\big((2N-1)!!(N+1)e^{\pi t/2}\eta^{\sigma-N}(x-n)^{-N}\epsilon^{-2N}(1+\epsilon)^{2N}\big),
\end{align*}
where the $O$-bound is uniform provided $N\geq2$.

Applying integration by parts once to the original expression for $_BI_2$ gives
\begin{multline*}
_BI_2=-\Big[e^{(x-n)z+it\log z}\Big(\tfrac{1}{n-x-\tfrac{it}{z}}\Big)D_N(z;n-x;\sigma,t)\Big]_{-i\infty}^{-i\eta} \\ +\int_{-i\infty}^{-i\eta}e^{(x-n)z+it\log z}D_{N+1}(z;n-x;\sigma,t)\,\mathrm{d}z.
\end{multline*}
Using the estimates we just derived for $D_N$ and (with $N$ replaced by $N+1$, so that our $N\geq2$ assumption becomes only $N\geq1$) for $_BI_2$, this becomes
\begin{align*}
\begin{split}
_BI_2&=O\bigg(\tfrac{e^{\pi t/2}z}{(n-x)z-it}(2N-1)!!(N+1)^2|z|^{\sigma-N-1}(x-n)^{-N}\epsilon^{-2N}(1+\epsilon)^{2N}\Big|_{z=-i\eta}\bigg) \\ &\hspace{4cm}+O\big((2N+1)!!(N+2)e^{\pi t/2}\eta^{\sigma-N-1}(x-n)^{-N-1}\epsilon^{-2N-2}(1+\epsilon)^{2N+2}\big)
\end{split} \\
\begin{split}
&=O\bigg(\tfrac{e^{\pi t/2}\eta}{\tfrac{\epsilon}{1+\epsilon}(x-n)\eta}(2N+1)!!(N+1)\eta^{\sigma-N-1}(x-n)^{-N}\epsilon^{-2N}(1+\epsilon)^{2N}\bigg) \\ &\hspace{4cm}+O\big((2N+1)!!(N+1)e^{\pi t/2}\eta^{\sigma-N-1}(x-n)^{-N-1}\epsilon^{-2N-2}(1+\epsilon)^{2N+2}\big)
\end{split} \\
&=O\Big((2N+1)!!(N+1)e^{\pi t/2}\eta^{\sigma-N-1}(x-n)^{-N-1}\epsilon^{-2N-1}(1+\epsilon)^{2N+1}\big[1+\tfrac{1+\epsilon}{\epsilon}\big]\Big),
\end{align*}
where we have again used the inequality $(x-n)\eta-t>\tfrac{\epsilon}{1+\epsilon}(x-n)\eta$. This gives us the required expression for $_BI_2$.

For $_BI_3$, we can use the same argument as in Case 1 of Lemma \ref{GU-DN-estimates} to estimate $D_N$. Starting from the expression (\ref{DN-series}) and using the fact that $|(n-x)z|\leq(x-n)\eta<(1-\epsilon)t$, we find
\begin{align*}
D_N(z;n-x;\sigma,t)&=O\Bigg((2N-1)!!|z|^{\sigma-1}\sum_{b=0}^N\sum_{c=0}^N\Big|\tfrac{(n-x)z}{t}\Big|^{N-b}|\sigma|^c\Big|\tfrac{t}{((n-x)z-it)^2}\Big|^N\Bigg) \\
&=O\bigg((2N-1)!!(N+1)^2|z|^{\sigma-1}\Big|\tfrac{t}{((n-x)z-it)^2}\Big|^N\bigg).
\end{align*}
Since $(n-x)z$ is non-negative imaginary with modulus at most $(1-\epsilon)t$, it follows that $\big|(n-x)z-it\big|>\epsilon t$, and thus
\begin{align*}
D_N&=O\Big((2N-1)!!(N+1)^2|z|^{\sigma-1}t^N(\epsilon t)^{-2N}\Big) \\
&=O\big((2N-1)!!(N+1)^2|z|^{\sigma-1}t^{-N}\epsilon^{-2N}\big).
\end{align*}
Therefore,
\begin{align*}
_BI_3&=O\bigg(\int_{0}^{-i\eta}e^{(x-n)z+it\log z}(2N-1)!!(N+1)^2|z|^{\sigma-1}t^{-N}\epsilon^{-2N}\,\mathrm{d}z\bigg) \\
&=O\Big((2N-1)!!(N+1)^2e^{\pi t/2}\big(\tfrac{\eta^{\sigma}}{\sigma}\big)t^{-N}\epsilon^{-2N}\Big).
\end{align*}
Applying integration by parts once to the original expression for $_BI_3$ gives
\begin{multline*}
_BI_3=-\Big[e^{(x-n)z+it\log z}\Big(\tfrac{1}{n-x-\tfrac{it}{z}}\Big)D_N(z;n-x;\sigma,t)\Big]_{0}^{-i\eta} \\ +\int_{0}^{-i\eta}e^{(x-n)z+it\log z}D_{N+1}(z;n-x;\sigma,t)\,\mathrm{d}z.
\end{multline*}
Using the estimates we just derived for $D_N$ and (with $N$ replaced by $N+1$) for $_BI_3$, this becomes
\begin{align*}
\begin{split}
_BI_3&=O\bigg(\tfrac{e^{\pi t/2}z}{(n-x)z-it}(2N-1)!!(N+1)^2|z|^{\sigma-1}t^{-N}\epsilon^{-2N}\Big|_{z=-i\eta}\bigg) \\ &\hspace{6cm}+O\big((2N+1)!!(N+2)^2e^{\pi t/2}\sigma^{-1}\eta^{\sigma}t^{-N-1}\epsilon^{-2N-2}\big)
\end{split} \\
&=O\bigg(\tfrac{e^{\pi t/2}\eta}{\epsilon t}(2N+1)!!(N+1)\eta^{\sigma-1}t^{-N}\epsilon^{-2N}\bigg)+O\big((2N+1)!!(N+1)^2e^{\pi t/2}\sigma^{-1}\eta^{\sigma}t^{-N-1}\epsilon^{-2N-2}\big) \\
&=O\big((2N+1)!!(N+1)^2e^{\pi t/2}\sigma^{-1}\eta^{\sigma}t^{-N-1}\epsilon^{-2N-2}\big),
\end{align*}
where we have used again the inequality $\big|(n-x)z-it\big|>\epsilon t$. This gives us the required expression for $_BI_3$.
\end{proof}

\begin{lem}
\label{GB-intDN-estimate3}
For $n>x$, the remainder term in (\ref{GB-summand-IbP4}) can be uniformly approximated by \[O\big((2N+1)!!(N+1)e^{\pi t/2}\eta^{\sigma-N-1}(n-x)^{\sigma-N-1}\big),\] and the remainder term in (\ref{GB-summand-IbP5}) can be uniformly approximated by \[O\big((2N+1)!!(N+1)^2e^{\pi t/2}\sigma^{-1}t^{\sigma-N-1}\big).\]
\end{lem}

\begin{proof}
Let $_BI_4$ and $_BI_5$ denote the two expressions we need to estimate, i.e.
\begin{equation*}
_BI_4\coloneqq\int_{-i\eta(n+1+x)}^{-i\eta(n-x)}e^{-w+it\log w}D_N(w;1;\sigma,t)\,\mathrm{d}w
\end{equation*}
and
\begin{equation*}
_BI_5\coloneqq\int_{-i\eta(n-x)}^{0}e^{-w+it\log w}D_N(w;1;\sigma,t)\,\mathrm{d}w.
\end{equation*}

By the first half (\ref{GL-DN-estimate1}) of Lemma \ref{GL-DN-estimate}, we have
\begin{equation}
\label{GB-DN-estimate4}
D_N(w;1;\sigma,t)=O\big((2N-1)!!(N+1)^2|w|^{\sigma-N-1}\big)
\end{equation}
for $\mathrm{Im}(w)<0$, which holds here since we have assumed $n-x>0$. Thus,
\begin{align}
\nonumber _BI_4&=O\bigg(\int_{-i\eta(n+1+x)}^{-i\eta(n-x)}e^{-w+it\log w}(2N-1)!!(N+1)^2|w|^{\sigma-N-1}\,\mathrm{d}w\bigg) \\
\label{GB-I4-estimate}
&=O\Big((2N-1)!!(N+1)e^{\pi t/2}\eta^{\sigma-N}\big[(n-x)^{\sigma-N}-(n+1+x)^{\sigma-N}\big]\Big),
\end{align}
where the $O$-bound is uniform provided $N\geq2$. Applying integration by parts once to the original expression for $_BI_4$ gives
\begin{multline*}
_BI_4=-\Big[e^{-w+it\log w}\Big(\tfrac{1}{1-\tfrac{it}{w}}\Big)D_N(w;1;\sigma,t)\Big]_{-i\eta(n+1+x)}^{-i\eta(n-x)}+\int_{-i\eta(n+1+x)}^{-i\eta(n-x)}e^{-w+it\log w}D_{N+1}(w;1;\sigma,t)\,\mathrm{d}w.
\end{multline*}
Using equation (\ref{GB-DN-estimate4}) for the first half of this expression, and equation (\ref{GB-I4-estimate}) (with $N$ replaced by $N+1$, so that our $N\geq2$ assumption becomes only $N\geq1$) for the second half, we find:
\begin{align*}
\begin{split}
_BI_4&=O\Big(e^{\pi t/2}\big[(2N-1)!!(N+1)^2|w|^{\sigma-N-1}\big]_{-i\eta(n+1+x)}^{-i\eta(n-x)}\Big) \\ &\hspace{3cm}+O\Big((2N+1)!!(N+2)e^{\pi t/2}\eta^{\sigma-N-1}\big[(n-x)^{\sigma-N-1}-(n+1+x)^{\sigma-N-1}\big]\Big)
\end{split} \\
&=O\Big((2N+1)!!(N+1)e^{\pi t/2}\eta^{\sigma-N-1}\big[(n-x)^{\sigma-N-1}-(n+1+x)^{\sigma-N-1}\big]\Big),
\end{align*}
which gives the required expression for $_BI_4$.

$_BI_5$ is slightly harder to estimate, since we need to split into two separate cases. By Lemma \ref{GL-DN-estimate}, we have \[D_N(w;1;\sigma,t)=O\Big((2N-1)!!(N+1)^2|w|^{\sigma-1}\max\big(|w|,t\big)^{-N}\Big)\] for $\mathrm{Im}(w)<0$, which again holds here by assumption. Thus,
\begin{align*}
_BI_5&=O\bigg(\int_{-i\eta(n-x)}^{0}e^{-w+it\log w}(2N-1)!!(N+1)^2|w|^{\sigma-1}\max\big(|w|,t\big)^{-N}\,\mathrm{d}w\bigg) \\
&=O\bigg((2N-1)!!(N+1)^2e^{\pi t/2}\int_{\eta(n-x)}^{0}w^{\sigma-1}\max(w,t)^{-N}\,\mathrm{d}w\bigg).
\end{align*}

\textbf{Case 1:} $\boldsymbol{\eta(n-x)\leq t}.$

In this case,
\begin{align}
\nonumber _BI_5&=O\bigg((2N-1)!!(N+1)^2e^{\pi t/2}\int_{\eta(n-x)}^{0}w^{\sigma-1}t^{-N}\,\mathrm{d}w\bigg) \\
\label{GB-I5-estimate1}
&=O\Big((2N-1)!!(N+1)^2e^{\pi t/2}\sigma^{-1}\eta^{\sigma}t^{-N}(n-x)^{\sigma}\Big).
\end{align}
Applying integration by parts once to the original expression for $_BI_5$ gives
\begin{equation*}
_BI_5=\Big[e^{-w+it\log w}\Big(\tfrac{1}{1-\tfrac{it}{w}}\Big)D_N(w;1;\sigma,t)\Big]_{w=-i\eta(n-x)}+\int_{\eta(n-x)}^{0}e^{-w+it\log w}D_{N+1}(w;1;\sigma,t)\,\mathrm{d}w.
\end{equation*}
Using (\ref{GL-DN-estimate2}) from Lemma \ref{GL-DN-estimate} for the first half of this expression, and (\ref{GB-I5-estimate1}) (with $N$ replaced by $N+1$) for the second half, we find:
\begin{align*}
\begin{split}
_BI_5&=O\Big(e^{\pi t/2}\Big(\tfrac{\eta(n-x)}{t}\Big)(2N-1)!!(N+1)^2[\eta(n-x)]^{\sigma-1}t^{-N}\Big) \\ &\hspace{7cm}+O\Big((2N+1)!!(N+2)^2e^{\pi t/2}\sigma^{-1}\eta^{\sigma}t^{-N-1}(n-x)^{\sigma}\Big)
\end{split} \\
&=O\Big((2N+1)!!(N+1)^2e^{\pi t/2}\sigma^{-1}[\eta(n-x)]^{\sigma}t^{-N-1}\Big) \\
&=O\big((2N+1)!!(N+1)^2e^{\pi t/2}\sigma^{-1}t^{\sigma-N-1}\big).
\end{align*}

\textbf{Case 2:} $\boldsymbol{\eta(n-x)>t}.$

In this case,
\begin{align}
\nonumber _BI_5&=O\bigg((2N-1)!!(N+1)^2e^{\pi t/2}\bigg[\int_{\eta(n-x)}^{t}w^{\sigma-N-1}\,\mathrm{d}w+\int_t^0w^{\sigma-1}t^{-N}\,\mathrm{d}w\bigg]\bigg) \\
\nonumber &=O\Big((2N-1)!!(N+1)e^{\pi t/2}\Big[t^{\sigma-N}-\big(\eta(x-n)\big)^{\sigma-N}+(N+1)\sigma^{-1}t^{\sigma-N}\Big]\Big) \\
\label{GB-I5-estimate2}
&=O\Big((2N-1)!!(N+1)^2e^{\pi t/2}\sigma^{-1}t^{\sigma-N}\Big),
\end{align}
where the $O$-bound is uniform provided $N\geq2$. Applying integration by parts as before, and then using (\ref{GL-DN-estimate1}) from Lemma \ref{GL-DN-estimate} for the first half of the resulting expression and (\ref{GB-I5-estimate2}) (with $N$ replaced by $N+1$, so that our $N\geq2$ assumption becomes only $N\geq1$) for the second half, we find:
\begin{align*}
_BI_5&=O\Big(e^{\pi t/2}(2N-1)!!(N+1)^2[\eta(n-x)]^{\sigma-N-1}\Big)+O\Big((2N+1)!!(N+2)^2e^{\pi t/2}\sigma^{-1}t^{\sigma-N-1}\Big) \\
&=O\big((2N+1)!!(N+1)^2e^{\pi t/2}\sigma^{-1}t^{\sigma-N-1}\big).
\end{align*}

In both cases, we have the required estimate for $_BI_5$.
\end{proof}

\begin{lem}
\label{GB-estimate}
We have the following uniform estimate for $G_B$:
\begin{multline*}
G_B(\sigma,t;\eta;x)=\frac{e^{-i\pi s/2}}{(2\pi)^s}\sum_{n=0}^{\big\lfloor x-\tfrac{t}{\eta}\big\rfloor}\frac{\Gamma(s)}{(x-n)^s}-\frac{e^{i\pi\sigma/2}e^{it\log\eta}}{(2\pi)^s}\sum_{n=0}^{M}\sum_{j=0}^{N-1}e^{i(n-x)\eta}\bigg[\Big(\tfrac{1}{n-x-\frac{it}{z}}\cdot\tfrac{\mathrm{d}}{\mathrm{d}z}\Big)^j\Big(\tfrac{z^{\sigma-1}}{n-x-\frac{it}{z}}\Big)\bigg]_{z=-i\eta} \\
+\frac{e^{i\pi\sigma/2}e^{it\log\eta}}{(2\pi)^s}\sum_{n=0}^{M}\sum_{j=0}^{N-1}e^{i(x+n+1)\eta}\bigg[\Big(\tfrac{1}{x+n+1-\frac{it}{z}}\cdot\tfrac{\mathrm{d}}{\mathrm{d}z}\Big)^j\Big(\tfrac{z^{\sigma-1}}{x+n+1-\frac{it}{z}}\Big)\bigg]_{z=-i\eta} \\
+O\Big((2N+1)!!(N+1)^2\sigma^{-1}\min(t,\eta)^{\sigma-N-1}\big[x^{1-\sigma} \\ +(x-\lfloor x\rfloor)^{-N-1}\big(\tfrac{1+\epsilon}{\epsilon}\big)^{2N+2}+x^{-\sigma}\big(\lfloor x\rfloor-x+1\big)^{-N-1}\big]\Big),
\end{multline*}
where $M$ is a finite number depending only on $N$, $x$, and $\eta$.
\end{lem}

\begin{proof}
We use Lemma \ref{GB-intDN-estimate1} to establish that for $n\leq x$, equation (\ref{GB-summand-IbP1}) becomes:
\begin{multline}
\label{GB-summand1}
\int_{-i\eta}^{0}e^{-(x+n+1)z}z^{s-1}\,\mathrm{d}z=e^{\pi t/2}e^{it\log\eta}\sum_{j=0}^{N-1}e^{i(x+n+1)\eta}\bigg[\Big(\tfrac{1}{x+n+1-\frac{it}{z}}\cdot\tfrac{\mathrm{d}}{\mathrm{d}z}\Big)^j\Big(\tfrac{z^{\sigma-1}}{x+n+1-\frac{it}{z}}\Big)\bigg]_{z=-i\eta} \\ +O\big((2N+1)!!(N+1)^2e^{\pi t/2}\sigma^{-1}t^{\sigma-N-1}(x+n+1)^{-\sigma}\big).
\end{multline}
We use Lemma \ref{GB-intDN-estimate2} to establish that for $n<x-(1+\epsilon)\tfrac{t}{\eta}$ and $x-(1-\epsilon)\tfrac{t}{\eta}<n\leq x$ respectively, equations (\ref{GB-summand-IbP2}) and (\ref{GB-summand-IbP3}) become:
\begin{multline}
\label{GB-summand2}
\int_{-i\infty}^{-i\eta}e^{(x-n)z}z^{s-1}\,\mathrm{d}z=-e^{\pi t/2}e^{it\log\eta}\sum_{j=0}^{N-1}e^{-i(x-n)\eta}\bigg[\Big(\tfrac{1}{n-x-\frac{it}{z}}\cdot\tfrac{\mathrm{d}}{\mathrm{d}z}\Big)^j\Big(\tfrac{z^{\sigma-1}}{n-x-\frac{it}{z}}\Big)\bigg]_{z=-i\eta} \\ +O\big((2N+1)!!(N+1)e^{\pi t/2}\eta^{\sigma-N-1}(x-n)^{-N-1}\epsilon^{-2N-2}(1+\epsilon)^{2N+2}\big)
\end{multline}
and
\begin{multline}
\label{GB-summand3}
\int_{0}^{-i\eta}e^{(x-n)z}z^{s-1}\,\mathrm{d}z=-e^{\pi t/2}e^{it\log\eta}\sum_{j=0}^{N-1}e^{-i(x-n)\eta}\Big(\tfrac{1}{n-x-\frac{it}{z}}\cdot\tfrac{\mathrm{d}}{\mathrm{d}z}\Big)^j\Big(\tfrac{z^{\sigma-1}}{n-x-\frac{it}{z}}\Big)\Bigg|_{z=-i\eta} \\ +O\big((2N+1)!!(N+1)^2e^{\pi t/2}\sigma^{-1}\eta^{\sigma}t^{-N-1}\epsilon^{-2N-2}\big).
\end{multline}
We use Lemma \ref{GB-intDN-estimate3} to establish that for $n>x$, equations (\ref{GB-summand-IbP4}) and (\ref{GB-summand-IbP5}) become:
\begin{multline*}
\int_{-i\eta(n+1+x)}^{-i\eta(n-x)}e^{-w}w^{s-1}\,\mathrm{d}w=e^{\pi t/2}e^{it\log\eta}\sum_{j=0}^{N-1}e^{i(n+1+x)\eta}(n+1+x)^{it}\bigg[\Big(\tfrac{1}{1-\frac{it}{w}}\cdot\tfrac{\mathrm{d}}{\mathrm{d}w}\Big)^j\Big(\tfrac{w^{\sigma-1}}{1-\frac{it}{w}}\Big)\bigg]_{w=-i\eta(n+1+x)} \\
-e^{\pi t/2}e^{it\log\eta}\sum_{j=0}^{N-1}e^{i(n-x)\eta}(n-x)^{it}\bigg[\Big(\tfrac{1}{1-\frac{it}{w}}\cdot\tfrac{\mathrm{d}}{\mathrm{d}w}\Big)^j\Big(\tfrac{w^{\sigma-1}}{1-\frac{it}{w}}\Big)\bigg]_{w=-i\eta(n-x)} \\
+O\big((2N+1)!!(N+1)e^{\pi t/2}\eta^{\sigma-N-1}(n-x)^{\sigma-N-1}\big)
\end{multline*}
and
\begin{multline*}
\int_{-i\eta(n-x)}^{0}e^{-w}w^{s-1}\,\mathrm{d}w=e^{\pi t/2}e^{it\log\eta}\sum_{j=0}^{N-1}e^{i(n-x)\eta}(n-x)^{it}\bigg[\Big(\tfrac{1}{1-\frac{it}{w}}\cdot\tfrac{\mathrm{d}}{\mathrm{d}w}\Big)^j\Big(\tfrac{w^{\sigma-1}}{1-\frac{it}{w}}\Big)\bigg]_{w=-i\eta(n-x)} \\
+O\big((2N+1)!!(N+1)^2e^{\pi t/2}\sigma^{-1}t^{\sigma-N-1}\big).
\end{multline*}
Now, re-substituting $w=(n+1+x)z$ or $w=(n-x)z$ as appropriate, the first of these two expressions becomes
\begin{align}
\begin{split}
\nonumber &e^{\pi t/2}e^{it\log\eta}\sum_{j=0}^{N-1}e^{i(n+1+x)\eta}(n+1+x)^{it}\bigg[\bigg(\tfrac{1/(n+1+x)}{1-\frac{it}{(n+1+x)z}}\cdot\tfrac{\mathrm{d}}{\mathrm{d}z}\bigg)^j\bigg(\tfrac{((n+1+x)z)^{\sigma-1}}{1-\frac{it}{(n+1+x)z}}\bigg)\bigg]_{z=-i\eta} \\
&\hspace{2cm}-e^{\pi t/2}e^{it\log\eta}\sum_{j=0}^{N-1}e^{i(n-x)\eta}(n-x)^{it}\bigg[\bigg(\tfrac{1/(n-x)}{1-\frac{it}{(n-x)z}}\cdot\tfrac{\mathrm{d}}{\mathrm{d}z}\bigg)^j\bigg(\tfrac{((n-x)z)^{\sigma-1}}{1-\frac{it}{(n-x)z}}\bigg)\bigg]_{z=-i\eta} \\
&\hspace{3.5cm}+O\big((2N+1)!!(N+1)e^{\pi t/2}\eta^{\sigma-N-1}(n-x)^{\sigma-N-1}\big)
\end{split} \\
\begin{split}
\label{GB-summand4}
=\;&e^{\pi t/2}e^{it\log\eta}(n+1+x)^s\sum_{j=0}^{N-1}e^{i(n+1+x)\eta}\bigg[\Big(\tfrac{1}{n+1+x-\frac{it}{z}}\cdot\tfrac{\mathrm{d}}{\mathrm{d}z}\Big)^j\Big(\tfrac{(z)^{\sigma-1}}{n+1+x-\frac{it}{z}}\Big)\bigg]_{z=-i\eta} \\
&\hspace{2cm}-e^{\pi t/2}e^{it\log\eta}(n-x)^s\sum_{j=0}^{N-1}e^{i(n-x)\eta}\bigg[\Big(\tfrac{1}{n-x-\frac{it}{z}}\cdot\tfrac{\mathrm{d}}{\mathrm{d}z}\Big)^j\Big(\tfrac{z^{\sigma-1}}{n-x-\frac{it}{z}}\Big)\bigg]_{z=-i\eta} \\
&\hspace{3cm}+O\big((2N+1)!!(N+1)e^{\pi t/2}\eta^{\sigma-N-1}(n-x)^{\sigma-N-1}\big).
\end{split}
\end{align}
Similarly, after re-substituting $w=(n-x)z$, the second expression becomes
\begin{multline}
\label{GB-summand5}
e^{\pi t/2}e^{it\log\eta}(n-x)^s\sum_{j=0}^{N-1}e^{i(n-x)\eta}\bigg[\Big(\tfrac{1}{n-x-\frac{it}{z}}\cdot\tfrac{\mathrm{d}}{\mathrm{d}z}\Big)^j\Big(\tfrac{z^{\sigma-1}}{n-x-\frac{it}{z}}\Big)\bigg]_{z=-i\eta} \\ +O\big((2N+1)!!(N+1)^2e^{\pi t/2}\sigma^{-1}t^{\sigma-N-1}\big).
\end{multline}

Substituting (\ref{GB-explicit}), (\ref{GB-summand1}), (\ref{GB-summand2}), (\ref{GB-summand3}), (\ref{GB-summand4}), (\ref{GB-summand5}) into the expression (\ref{GB-prelemma}) for $G_B$ and noting the cancellation of certain terms originating from (\ref{GB-summand4}) and (\ref{GB-summand5}), we find the following expression for $G_B$:
\begin{align*}
\begin{split}
G_B&=\frac{e^{-i\pi s/2}}{(2\pi)^s}\sum_{n=0}^{\big\lfloor x-\tfrac{t}{\eta}\big\rfloor}\frac{\Gamma(s)}{(x-n)^s} \\
&\hspace{2cm}+\frac{e^{i\pi\sigma/2}}{(2\pi)^s}\sum_{n=0}^{\lfloor x\rfloor}\Bigg(e^{it\log\eta}\sum_{j=0}^{N-1}e^{i(x+n+1)\eta}\bigg[\Big(\tfrac{1}{x+n+1-\frac{it}{z}}\cdot\tfrac{\mathrm{d}}{\mathrm{d}z}\Big)^j\Big(\tfrac{z^{\sigma-1}}{x+n+1-\frac{it}{z}}\Big)\bigg]_{z=-i\eta}\Bigg) \\
&\hspace{2cm}-\frac{e^{i\pi\sigma/2}}{(2\pi)^s}\sum_{n=0}^{\lfloor x\rfloor}\Bigg(e^{it\log\eta}\sum_{j=0}^{N-1}e^{-i(x-n)\eta}\bigg[\Big(\tfrac{1}{n-x-\frac{it}{z}}\cdot\tfrac{\mathrm{d}}{\mathrm{d}z}\Big)^j\Big(\tfrac{z^{\sigma-1}}{n-x-\frac{it}{z}}\Big)\bigg]_{z=-i\eta}\Bigg) \\
&\hspace{2cm}+\frac{e^{i\pi\sigma/2}}{(2\pi)^s}\sum_{n=\lfloor x\rfloor+1}^{\infty}\Bigg(e^{it\log\eta}\sum_{j=0}^{N-1}e^{i(n+1+x)\eta}\bigg[\Big(\tfrac{1}{n+1+x-\frac{it}{z}}\cdot\tfrac{\mathrm{d}}{\mathrm{d}z}\Big)^j\Big(\tfrac{(z)^{\sigma-1}}{n+1+x-\frac{it}{z}}\Big)\bigg]_{z=-i\eta}\Bigg) \\
&\hspace{2cm}-\frac{e^{i\pi\sigma/2}}{(2\pi)^s}\sum_{n=\lfloor x\rfloor+1}^{\infty}\Bigg(e^{it\log\eta}\sum_{j=0}^{N-1}e^{i(n-x)\eta}\bigg[\Big(\tfrac{1}{n-x-\frac{it}{z}}\cdot\tfrac{\mathrm{d}}{\mathrm{d}z}\Big)^j\Big(\tfrac{z^{\sigma-1}}{n-x-\frac{it}{z}}\Big)\bigg]_{z=-i\eta}\Bigg) \\
&\hspace{2cm}+O\bigg(\sum_{n=0}^{\lfloor x\rfloor}(2N+1)!!(N+1)^2\sigma^{-1}t^{\sigma-N-1}(x+n+1)^{-\sigma}\bigg) \\
&\hspace{2cm}+O\Bigg(\sum_{n=0}^{\big\lfloor x-\tfrac{t}{\eta}\big\rfloor}(2N+1)!!(N+1)\eta^{\sigma-N-1}(x-n)^{-N-1}\epsilon^{-2N-2}(1+\epsilon)^{2N+2}\Bigg) \\
&\hspace{2cm}+O\Bigg(\sum_{n=\big\lceil x-\tfrac{t}{\eta}\big\rceil}^{\lfloor x\rfloor}(2N+1)!!(N+1)^2\sigma^{-1}\eta^{\sigma}t^{-N-1}\epsilon^{-2N-2}\Bigg) \\
&\hspace{2cm}+O\bigg(\sum_{n=\lfloor x\rfloor+1}^{\infty}(n+1+x)^{-s}(2N+1)!!(N+1)\eta^{\sigma-N-1}(n-x)^{\sigma-N-1}\bigg) \\
&\hspace{2cm}+O\bigg(\sum_{n=\lfloor x\rfloor+1}^{\infty}\Big((n+1+x)^{-s}-(n-x)^{-s}\Big)(2N+1)!!(N+1)^2\sigma^{-1}t^{\sigma-N-1}\bigg).
\end{split}
\end{align*}
Let us consider each of the remainder terms in turn, with the $O$-bound in each case being uniform. First, \[\sum_{n=0}^{\lfloor x\rfloor}(x+n+1)^{-\sigma}=O\big(x^{1-\sigma}\big);\] this follows by splitting the series into sums from $\tfrac{x}{2^k}$ to $\tfrac{x}{2^{k-1}}$ for $1\leq k\leq\log_2x$. Second, \[\sum_{n=0}^{\big\lfloor x-\tfrac{t}{\eta}\big\rfloor}(x-n)^{-N-1}=O\big((x-\lfloor x\rfloor)^{-N-1}\big)+O(1).\] Third, if $\tfrac{t}{\eta}<1$, the sum \[\sum_{n=\big\lceil x-\tfrac{t}{\eta}\big\rceil}^{\lfloor x\rfloor}1\] is non-existent, while if $\tfrac{t}{\eta}\geq1$ it is $O\big(\tfrac{t}{\eta})$. Fourth, \[\sum_{n=\lfloor x\rfloor+1}^{\infty}(n+1+x)^{-s}(n-x)^{\sigma-N-1}=O\Big(x^{-\sigma}\big(\lfloor x\rfloor-x+1\big)^{\sigma-N-1}+1\Big).\] And finally,
\begin{align*}
\sum_{n=\lfloor x\rfloor+1}^{\infty}\Big((n+1+x)^{-s}-(n-x)^{-s}\Big)&=\sum_{n=\lfloor x\rfloor+1}^{\infty}n^{-\sigma}O\Big((1+\tfrac{1+x}{n})^{-\sigma}-(1-\tfrac{x}{n})^{-\sigma}\Big) \\
&=\sum_{n=\lfloor x\rfloor+1}^{\infty}\Big(O(n^{-\sigma-1})+O(n^{-\sigma-2})+\dots\Big) \\
&=O(1).
\end{align*}
Substituting all of the above estimates into the expression for $G_B$ gives:
\begin{align*}
\begin{split}
G_B&=\frac{e^{-i\pi s/2}}{(2\pi)^s}\sum_{n=0}^{\big\lfloor x-\tfrac{t}{\eta}\big\rfloor}\frac{\Gamma(s)}{(x-n)^s}-\frac{e^{i\pi\sigma/2}e^{it\log\eta}}{(2\pi)^s}\sum_{n=0}^{\infty}\sum_{j=0}^{N-1}e^{i(n-x)\eta}\bigg[\Big(\tfrac{1}{n-x-\frac{it}{z}}\cdot\tfrac{\mathrm{d}}{\mathrm{d}z}\Big)^j\Big(\tfrac{z^{\sigma-1}}{n-x-\frac{it}{z}}\Big)\bigg]_{z=-i\eta} \\
&\hspace{2cm}+\frac{e^{i\pi\sigma/2}e^{it\log\eta}}{(2\pi)^s}\sum_{n=0}^{\infty}\sum_{j=0}^{N-1}e^{i(x+n+1)\eta}\bigg[\Big(\tfrac{1}{x+n+1-\frac{it}{z}}\cdot\tfrac{\mathrm{d}}{\mathrm{d}z}\Big)^j\Big(\tfrac{z^{\sigma-1}}{x+n+1-\frac{it}{z}}\Big)\bigg]_{z=-i\eta} \\
&\hspace{2cm}+O\Big((2N+1)!!(N+1)^2\sigma^{-1}t^{\sigma-N-1}x^{1-\sigma}\Big) \\
&\hspace{2cm}+O\Big((2N+1)!!(N+1)\eta^{\sigma-N-1}(x-\lfloor x\rfloor)^{-N-1}\epsilon^{-2N-2}(1+\epsilon)^{2N+2}\Big) \\
&\hspace{2cm}+O\Big((2N+1)!!(N+1)^2\sigma^{-1}\eta^{\sigma-1}t^{-N}\epsilon^{-2N-2}\Big) \\
&\hspace{2cm}+O\Big((2N+1)!!(N+1)\eta^{\sigma-N-1}x^{-\sigma}\big(\lfloor x\rfloor-x+1\big)^{\sigma-N-1}\Big) \\
&\hspace{2cm}+O\Big((2N+1)!!(N+1)^2\sigma^{-1}t^{\sigma-N-1}\Big).
\end{split}
\end{align*}
Each of the five error terms can be approximated by $(2N+1)!!(N+1)^2\sigma^{-1}$ times either $t^{\sigma-N-1}$ or $\eta^{\sigma-N-1}$ times one of $x^{1-\sigma}$, $(x-\lfloor x\rfloor)^{-N-1}\epsilon^{-2N-2}(1+\epsilon)^{2N+2}$, and $x^{-\sigma}\big(\lfloor x\rfloor-x+1\big)^{\sigma-N-1}$. Thus, we finally get the required form of the error terms. It remains to prove that the infinite series over $n$ can be reduced to a finite one by finding an appropriate upper bound $M(N,x,\eta)$: in other words, to find $M$ large enough that the series
\begin{equation}
\label{GB-M-series1}
\sum_{n=M+1}^{\infty}\sum_{j=0}^{N-1}e^{i(n-x)\eta}\bigg[\Big(\tfrac{1}{n-x-\frac{it}{z}}\cdot\tfrac{\mathrm{d}}{\mathrm{d}z}\Big)^j\Big(\tfrac{z^{\sigma-1}}{n-x-\frac{it}{z}}\Big)\bigg]_{z=-i\eta}
\end{equation}
and
\begin{equation}
\label{GB-M-series2}
\sum_{n=M+1}^{\infty}\sum_{j=0}^{N-1}e^{i(x+n+1)\eta}\bigg[\Big(\tfrac{1}{x+n+1-\frac{it}{z}}\cdot\tfrac{\mathrm{d}}{\mathrm{d}z}\Big)^j\Big(\tfrac{z^{\sigma-1}}{x+n+1-\frac{it}{z}}\Big)\bigg]_{z=-i\eta}
\end{equation}
can be swallowed up by the existing remainder term.

We assume
\begin{equation}
\label{GB-M-condition1}
M>x,
\end{equation}
so that both series above can be estimated using the definition of $D_N$ and the result (\ref{GL-DN-estimate1}) from Lemma \ref{GL-DN-estimate}. The expressions (\ref{GB-M-series1}) and (\ref{GB-M-series2}) can be rewritten as
\begin{align*}
&{\color{white}=}\sum_{n=M+1}^{\infty}\sum_{j=0}^{N-1}\frac{e^{i(n-x)\eta}}{n-x+\frac{t}{\eta}}D_j(-i\eta;n-x;\sigma,t) \\
&=\sum_{n=M+1}^{\infty}\frac{e^{i(n-x)\eta}(-i\eta)^{\sigma-1}}{n-x+\frac{t}{\eta}}+\sum_{n=M+1}^{\infty}\sum_{j=1}^{N-1}O\bigg(\frac{1}{n-x+\frac{t}{\eta}}(2j-1)!!(j+1)^2\eta^{\sigma-j-1}(n-x)^{-j}\bigg) \\
&=O\bigg(\eta^{\sigma-1}\sum_{n=M+1}^{\infty}\frac{e^{in\eta}}{n-x}\bigg)+\sum_{j=1}^{N-1}O\bigg((2j-1)!!(j+1)^2\eta^{\sigma-j-1}\sum_{n=M+1}^{\infty}(n-x)^{-j-1}\bigg)
\end{align*}
and
\begin{align*}
&{\color{white}=}\sum_{n=M+1}^{\infty}\sum_{j=0}^{N-1}\frac{e^{i(x+n+1)\eta}}{x+n+1+\frac{t}{\eta}}D_j(-i\eta;x+n+1;\sigma,t) \\
&=\sum_{n=M+1}^{\infty}\frac{e^{i(x+n+1)\eta}(-i\eta)^{\sigma-1}}{x+n+1+\frac{t}{\eta}}+\sum_{n=M+1}^{\infty}\sum_{j=1}^{N-1}O\bigg(\frac{1}{x+n+1+\frac{t}{\eta}}(2j-1)!!(j+1)^2\eta^{\sigma-j-1}(x+n+1)^{-j}\bigg) \\
&=O\bigg(\eta^{\sigma-1}\sum_{n=M+1}^{\infty}\frac{e^{in\eta}}{x+n+1}\bigg)+\sum_{j=1}^{N-1}O\bigg((2j-1)!!(j+1)^2\eta^{\sigma-j-1}\sum_{n=M+1}^{\infty}(x+n+1)^{-j-1}\bigg).
\end{align*}
Both of these can be bounded by $O\big((2N-1)!!(N+1)^2\eta^{\sigma-N-1}\big)$ just as in Lemma \ref{GL-estimate}, provided that $M$ satisfies the conditions
\begin{equation}
\label{GB-M-condition2}
\sum_{n=M+1}^{\infty}e^{in\eta}(n-x)^{-1}\leq\eta^{-N}
\end{equation}
and
\begin{equation}
\label{GB-M-condition3}
\sum_{n=M+1}^{\infty}(n-x)^{-2}\leq\eta^{1-N}.
\end{equation}

Thus, for any $M$ satisfying the conditions (\ref{GB-M-condition1}) and (\ref{GB-M-condition2}) and (\ref{GB-M-condition3}), the remainder term from the tail of the $n$-series is swallowed up by the other remainder terms, and we have the desired result.
\end{proof}

\section{Final result}
\begin{thm}
\label{zeta-estimate}
The modified Hurwitz zeta function is given by the following finite asymptotic series:
\begin{multline}
\label{zeta-estimate-formula}
\zeta_1(x,s)=\sum_{n=1}^{\big\lfloor\tfrac{t}{\eta}-x\big\rfloor}\frac{1}{(x+n)^s}-\sum_{n=0}^{\big\lfloor x-\tfrac{t}{\eta}\big\rfloor}\frac{1}{(x-n)^s}+\chi(s)\Bigg(\sum_{m=1}^{\lfloor\eta/2\pi\rfloor}e^{-2\pi imx}m^{s-1} \\ +\frac{e^{-i\pi\sigma/2}e^{it\log\eta}}{(2\pi)^s}\sum_{n=1}^{M}\sum_{j=0}^{N-1}e^{-i(x+n)\eta}\bigg[\Big(\tfrac{1}{x+n-\frac{it}{z}}\cdot\tfrac{\mathrm{d}}{\mathrm{d}z}\Big)^j\Big(\tfrac{z^{\sigma-1}}{x+n-\frac{it}{z}}\Big)\bigg]_{z=i\eta} \\ +\frac{e^{i\pi\sigma/2}e^{it\log\eta}}{(2\pi)^s}\sum_{n=0}^{M}\sum_{j=0}^{N-1}e^{i(n-x)\eta}\bigg[\Big(\tfrac{1}{n-x-\frac{it}{z}}\cdot\tfrac{\mathrm{d}}{\mathrm{d}z}\Big)^j\Big(\tfrac{z^{\sigma-1}}{n-x-\frac{it}{z}}\Big)\bigg]_{z=-i\eta} \\ +O\Big((2N+1)!!(N+1)^2\sigma^{-1}\min(t,\eta)^{\sigma-N-1}x^{-\sigma}K_N(x)\big(\tfrac{1+\epsilon}{\epsilon}\big)^{2N+2}\Big)\Bigg),
\end{multline}
where $s=\sigma+it$, $0<\sigma\leq1$, $0<t<\infty$, $0<x<\infty$, $0<\eta<\infty$ satisfies Assumption \ref{Assumption} for some fixed $\epsilon>0$, $M$ is a natural number depending only on $N$, $x$, and $\eta$, the function $K_N(x)$ is defined by \[K_N(x)=\max\Big(x,(x-\lfloor x\rfloor)^{-N-1},(\lfloor x\rfloor-x+1)^{-N-1}\Big),\] and the $O$-constant is uniform in all variables.
\end{thm}

\begin{proof}
Substituting the results of Lemma \ref{GL-estimate}, Lemma \ref{GU-estimate}, and Lemma \ref{GB-estimate} into the identity (\ref{Hurwitz-eta-formula}), we find:
\begin{align*}
\begin{split}
\zeta_1(x,s)&=\chi(s)\Bigg(\sum_{m=1}^{\lfloor\eta/2\pi\rfloor}e^{-2\pi imx}m^{s-1}+\frac{e^{-i\pi s/2}}{(2\pi)^s}\sum_{n=1}^{\big\lfloor\tfrac{t}{\eta}-x\big\rfloor}\frac{\Gamma(s)}{(x+n)^s}-\frac{e^{-i\pi s/2}}{(2\pi)^s}\sum_{n=0}^{\big\lfloor x-\tfrac{t}{\eta}\big\rfloor}\frac{\Gamma(s)}{(x-n)^s} \\
&\hspace{2cm}+\frac{e^{i\pi\sigma/2}e^{it\log\eta}}{(2\pi)^s}\sum_{n=1}^{M}\sum_{j=0}^{N-1}e^{i(x+n)\eta}\Bigg[\Big(\tfrac{1}{x+n-\frac{it}{z}}\cdot\tfrac{\mathrm{d}}{\mathrm{d}z}\Big)^j\Big(\tfrac{z^{\sigma-1}}{x+n-\frac{it}{z}}\Big)\Bigg]_{z=-i\eta} \\
&\hspace{2cm}+\frac{e^{-i\pi\sigma/2}e^{it\log\eta}}{(2\pi)^s}\sum_{n=1}^{M}\sum_{j=0}^{N-1}e^{-i(x+n)\eta}\bigg[\Big(\tfrac{1}{x+n-\frac{it}{z}}\cdot\tfrac{\mathrm{d}}{\mathrm{d}z}\Big)^j\Big(\tfrac{z^{\sigma-1}}{x+n-\frac{it}{z}}\Big)\bigg]_{z=i\eta} \\
&\hspace{2cm}+\frac{e^{i\pi\sigma/2}e^{it\log\eta}}{(2\pi)^s}\sum_{n=0}^{M}\sum_{j=0}^{N-1}e^{i(n-x)\eta}\bigg[\Big(\tfrac{1}{n-x-\frac{it}{z}}\cdot\tfrac{\mathrm{d}}{\mathrm{d}z}\Big)^j\Big(\tfrac{z^{\sigma-1}}{n-x-\frac{it}{z}}\Big)\bigg]_{z=-i\eta} \\
&\hspace{2cm}-\frac{e^{i\pi\sigma/2}e^{it\log\eta}}{(2\pi)^s}\sum_{n=0}^{M}\sum_{j=0}^{N-1}e^{i(x+n+1)\eta}\bigg[\Big(\tfrac{1}{x+n+1-\frac{it}{z}}\cdot\tfrac{\mathrm{d}}{\mathrm{d}z}\Big)^j\Big(\tfrac{z^{\sigma-1}}{x+n+1-\frac{it}{z}}\Big)\bigg]_{z=-i\eta} \\
&\hspace{2cm}+O\Big((2N-1)!!(N+1)^2\eta^{\sigma-N-1}\Big) \\ 
&\hspace{2cm}+O\big((2N+1)!!(N+1)^2\sigma^{-1}\eta^{\sigma-N-1}\epsilon^{-2N-2}(1+\epsilon)^{2N+2}\big) \\
&\hspace{2cm}+O\Big((2N+1)!!(N+1)^2\sigma^{-1}\min(t,\eta)^{\sigma-N-1}\big[x^{1-\sigma} \\
&\hspace{5cm}+(x-\lfloor x\rfloor)^{-N-1}\big(\tfrac{1+\epsilon}{\epsilon}\big)^{2N+2}+x^{\sigma}\big(\lfloor x\rfloor-x+1\big)^{-N-1}\big]\Big)\Bigg)
\end{split} \\
\begin{split}
&=\chi(s)\Bigg(\sum_{m=1}^{\lfloor\eta/2\pi\rfloor}e^{-2\pi imx}m^{s-1}+\frac{e^{-i\pi s/2}}{(2\pi)^s}\sum_{n=1}^{\big\lfloor\tfrac{t}{\eta}-x\big\rfloor}\frac{\Gamma(s)}{(x+n)^s}-\frac{e^{-i\pi s/2}}{(2\pi)^s}\sum_{n=0}^{\big\lfloor x-\tfrac{t}{\eta}\big\rfloor}\frac{\Gamma(s)}{(x-n)^s} \\
&\hspace{2cm}+\frac{e^{-i\pi\sigma/2}e^{it\log\eta}}{(2\pi)^s}\sum_{n=1}^{M}\sum_{j=0}^{N-1}e^{-i(x+n)\eta}\bigg[\Big(\tfrac{1}{x+n-\frac{it}{z}}\cdot\tfrac{\mathrm{d}}{\mathrm{d}z}\Big)^j\Big(\tfrac{z^{\sigma-1}}{x+n-\frac{it}{z}}\Big)\bigg]_{z=i\eta} \\
&\hspace{2cm}+\frac{e^{i\pi\sigma/2}e^{it\log\eta}}{(2\pi)^s}\sum_{n=0}^{M}\sum_{j=0}^{N-1}e^{i(n-x)\eta}\bigg[\Big(\tfrac{1}{n-x-\frac{it}{z}}\cdot\tfrac{\mathrm{d}}{\mathrm{d}z}\Big)^j\Big(\tfrac{z^{\sigma-1}}{n-x-\frac{it}{z}}\Big)\bigg]_{z=-i\eta} \\
&\hspace{2cm}+O\Big((2N+1)!!(N+1)^2\sigma^{-1}\min(t,\eta)^{\sigma-N-1}\big[x^{1-\sigma} \\
&\hspace{5cm}+(x-\lfloor x\rfloor)^{-N-1}\big(\tfrac{1+\epsilon}{\epsilon}\big)^{2N+2}+x^{-\sigma}\big(\lfloor x\rfloor-x+1\big)^{-N-1}\big]\Big)\Bigg).
\end{split}
\end{align*}

The second and third series in this expression can be simplified by using (\ref{chi}) together with Euler's reflection formula $\Gamma(1-s)\Gamma(s)=\frac{\pi}{\sin(\pi s)}\;,\;s\in\mathbb{C}\backslash\mathbb{Z},$ to obtain:
\begin{align*}
\chi(s)\tfrac{e^{-i\pi s/2}}{(2\pi)^s}\Gamma(s)&=\tfrac{(2\pi)^s}{\pi}\Gamma(1-s)\sin\big(\tfrac{\pi s}{2}\big)\tfrac{e^{-i\pi s/2}}{(2\pi)^s}\Gamma(s)=\tfrac{e^{-i\pi s/2}}{\pi}\sin\big(\tfrac{\pi s}{2}\big)\tfrac{\pi}{\sin(\pi s)} \\
&=\frac{e^{-i\pi s/2}}{2\cos\big(\tfrac{\pi s}{2}\big)}=\frac{e^{-i\pi s/2}}{e^{i\pi s/2}+e^{-i\pi s/2}}=\frac{1}{1+e^{i\pi\sigma}e^{-\pi t}}=1+O\big(e^{-\pi t}\big).
\end{align*}
Since exponential decay in $t$ is negligible in the large-$t$ asymptotics considered here, we can therefore rewrite the above formula for $\zeta_1(x,s)$ as follows:
\begin{multline*}
\zeta_1(x,s)=\sum_{n=1}^{\big\lfloor\tfrac{t}{\eta}-x\big\rfloor}\frac{1}{(x+n)^s}-\sum_{n=0}^{\big\lfloor x-\tfrac{t}{\eta}\big\rfloor}\frac{1}{(x-n)^s}+\chi(s)\Bigg(\sum_{m=1}^{\lfloor\eta/2\pi\rfloor}e^{-2\pi imx}m^{s-1} \\
+\frac{e^{-i\pi\sigma/2}e^{it\log\eta}}{(2\pi)^s}\sum_{n=1}^{M}\sum_{j=0}^{N-1}e^{-i(x+n)\eta}\bigg[\Big(\tfrac{1}{x+n-\frac{it}{z}}\cdot\tfrac{\mathrm{d}}{\mathrm{d}z}\Big)^j\Big(\tfrac{z^{\sigma-1}}{x+n-\frac{it}{z}}\Big)\bigg]_{z=i\eta} \\
+\frac{e^{i\pi\sigma/2}e^{it\log\eta}}{(2\pi)^s}\sum_{n=0}^{M}\sum_{j=0}^{N-1}e^{i(n-x)\eta}\bigg[\Big(\tfrac{1}{n-x-\frac{it}{z}}\cdot\tfrac{\mathrm{d}}{\mathrm{d}z}\Big)^j\Big(\tfrac{z^{\sigma-1}}{n-x-\frac{it}{z}}\Big)\bigg]_{z=-i\eta} \\
+O\Big((2N+1)!!(N+1)^2\sigma^{-1}\min(t,\eta)^{\sigma-N-1}\big[x^{1-\sigma} \\
+(x-\lfloor x\rfloor)^{-N-1}\big(\tfrac{1+\epsilon}{\epsilon}\big)^{2N+2}+x^{-\sigma}\big(\lfloor x\rfloor-x+1\big)^{-N-1}\big]\Big)\Bigg).
\end{multline*}
Then the required result (\ref{zeta-estimate-formula}) follows. The dependence of $M$ on $N$, $x$, and $\eta$ is given by the equations (\ref{GB-M-condition1}), (\ref{GB-M-condition2}), and (\ref{GB-M-condition3}), since the previous conditions (\ref{GL-M-condition1}) and (\ref{GL-M-condition2}) are implied by these.
\end{proof}

Note that (\ref{zeta-estimate-formula}) certainly describes a valid asymptotic series -- each term in the series over $j$ being smaller than the last, and the remainder term smaller than all of them -- precisely because the result is valid for all $N$. The remainder term is of an order in $t$ which decreases as $N$ increases, and reducing $N$ is equivalent to removing terms from the end of the series, so the estimate for the remainder term gives us the order of each term in the series over $j$.

\begin{coroll}
With all notation and assumptions as in Theorem \ref{zeta-estimate}, the leading-order asymptotics for the Hurwitz zeta function can be expressed by the formulae below.

\textbf{Case 1:} if $\eta>\frac{t}{x}$, then
\begin{equation}
\label{zeta-1st-order1}
\zeta_1(x,s)\sim-\sum_{n=0}^{\big\lfloor x-\tfrac{t}{\eta}\big\rfloor}\frac{1}{(x-n)^s}+\chi(s)\sum_{m=1}^{\lfloor\eta/2\pi\rfloor}e^{-2\pi imx}m^{s-1}.
\end{equation}

\textbf{Case 2:} if $\eta<\frac{t}{x+1}$, then
\begin{equation}
\label{zeta-1st-order2}
\zeta_1(x,s)\sim\sum_{n=1}^{\big\lfloor\tfrac{t}{\eta}-x\big\rfloor}\frac{1}{(x+n)^s}+\chi(s)\sum_{m=1}^{\lfloor\eta/2\pi\rfloor}e^{-2\pi imx}m^{s-1}.
\end{equation}

\textbf{Case 3:} if $\frac{t}{x+1}<\eta<\frac{t}{x}$, then
\begin{equation}
\label{zeta-1st-order3}
\zeta_1(x,s)\sim\chi(s)\sum_{m=1}^{\lfloor\eta/2\pi\rfloor}e^{-2\pi imx}m^{s-1}.
\end{equation}
\end{coroll}

\begin{proof}
If $\eta>\frac{t}{x}$, then $\frac{t}{\eta}-x<0$, thus the first sum in (\ref{zeta-estimate-formula}) vanishes.

If $\eta<\frac{t}{x+1}$, then $x-\frac{t}{\eta}<-1$, thus the second sum in (\ref{zeta-estimate-formula}) vanishes.

If $\frac{t}{x+1}<\eta<\frac{t}{x}$, then $\frac{t}{\eta}-x<1$ and $x-\frac{t}{\eta}<0$, thus the first and second sums in (\ref{zeta-estimate-formula}) both vanish.
\end{proof}

By Assumption \ref{Assumption}, $\eta$ cannot be equal to either $\frac{t}{x}$ or $\frac{t}{x+1}$. Thus, the three cases above cover \textit{all} possibilities for $\eta$.

\begin{remark}
\normalfont Let us consider the particular value $x=0$, and compare the results of Theorem \ref{zeta-estimate} with the results for $\zeta(s)$ obtained in \cite{fokas-lenells}.

The three cases considered in the above corollary correspond to the cases into which the problem was separated in \cite{fokas-lenells}. Case 1 above is not possible when $x=0$, but Case 2 above now becomes the $\eta<t$ case of \cite{fokas-lenells}, and Case 3 above becomes the $\eta>t$ case of \cite{fokas-lenells}. The case $\eta=t$, which is considered in Theorem 3.2 of \cite{fokas-lenells}, is prohibited when $x=0$, under the terms of Assumption \ref{Assumption}.

When $\eta<t$, the formulae for $\zeta(s)$ obtained in Theorems 4.1 and 4.4 of \cite{fokas-lenells} were derived by entirely different methods from those used here, so it would be difficult to compare them with the result of our Theorem \ref{zeta-estimate} without reducing the relevant expressions all the way back to the original form of $\zeta_1(x,s)$. However, we can easily check that the leading-order terms of the expressions derived in \cite{fokas-lenells} are identical with those of Theorem \ref{zeta-estimate}: both the expressions proven in Theorems 4.1 and 4.4 of \cite{fokas-lenells} yield \[\zeta(s)\sim\sum_{n=1}^{\big\lfloor\tfrac{t}{\eta}\big\rfloor}\frac{1}{n^s}+\chi(s)\sum_{m=1}^{\lfloor\eta/2\pi\rfloor}m^{s-1},\] and this is precisely the expression obtained from (\ref{zeta-1st-order2}) under the assumption that $x=0$.

When $\eta>(1+\epsilon)t$ for some $\epsilon>0$, the formula for $\zeta(s)$ obtained in Theorem 3.1 of \cite{fokas-lenells} can be written as follows:
\begin{multline}
\label{zeta-estimate-case3}
\zeta(1-s)=\sum_{n=1}^{\big\lfloor\tfrac{\eta}{2\pi}\big\rfloor}n^{s-1}-\frac{1}{s}\Big(\frac{\eta}{2\pi}\Big)^s+\frac{e^{i\pi s/2}}{(2\pi)^s}\sum_{n=1}^{M}\sum_{j=0}^{N-1}e^{-nz+it\log z}\bigg[\Big(\tfrac{1}{n-\frac{it}{z}}\cdot\tfrac{\mathrm{d}}{\mathrm{d}z}\Big)^j\Big(\tfrac{z^{\sigma-1}}{n-\frac{it}{z}}\Big)\bigg]_{z=-i\eta} \\ +\frac{e^{-i\pi s/2}}{(2\pi)^s}\sum_{n=1}^{M}\sum_{j=0}^{N-1}e^{-nz+it\log z}\bigg[\Big(\tfrac{1}{n-\frac{it}{z}}\cdot\tfrac{\mathrm{d}}{\mathrm{d}z}\Big)^j\Big(\tfrac{z^{\sigma-1}}{n-\frac{it}{z}}\Big)\bigg]_{z=i\eta} \\ +O\Big((2N+1)!!N\big(\tfrac{1+\epsilon}{\epsilon}\big)^{2(N+1)}\eta^{\sigma-N-1}\Big).
\end{multline}
On the other hand, our result (\ref{zeta-estimate-formula}), with $x=0$ and $\eta>t$, becomes:
\begin{multline}
\label{zeta-estimate-formula-zero}
\zeta(s)=\chi(s)\Bigg(\sum_{m=1}^{\lfloor\eta/2\pi\rfloor}m^{s-1}+\frac{e^{-i\pi\sigma/2}e^{it\log\eta}}{(2\pi)^s}\sum_{n=1}^{M}\sum_{j=0}^{N-1}e^{-in\eta}\bigg[\Big(\tfrac{1}{n-\frac{it}{z}}\cdot\tfrac{\mathrm{d}}{\mathrm{d}z}\Big)^j\Big(\tfrac{z^{\sigma-1}}{n-\frac{it}{z}}\Big)\bigg]_{z=i\eta} \\ +\frac{e^{i\pi\sigma/2}e^{it\log\eta}}{(2\pi)^s}\sum_{n=0}^{M}\sum_{j=0}^{N-1}e^{in\eta}\bigg[\Big(\tfrac{1}{n-\frac{it}{z}}\cdot\tfrac{\mathrm{d}}{\mathrm{d}z}\Big)^j\Big(\tfrac{z^{\sigma-1}}{n-\frac{it}{z}}\Big)\bigg]_{z=-i\eta} \\ +O\Big((2N+1)!!(N+1)^2t^{\sigma-N-1}\big(\tfrac{1+\epsilon}{\epsilon}\big)^{2N+2}\Big)\Bigg).
\end{multline}
Note that our assumption $\eta>(1+\epsilon)t$ guarantees that Assumption \ref{Assumption} is valid, because $1>(1+\epsilon)\tfrac{t}{\eta}$ and $0<(1-\epsilon)\tfrac{t}{\eta}$. Using the well-known identity $\zeta(s)=\chi(s)\zeta(1-s)$, it is straightforward to check that equations (\ref{zeta-estimate-case3}) and (\ref{zeta-estimate-formula-zero}) are equivalent. In particular, the term \[-\frac{1}{s}\Big(\frac{\eta}{2\pi}\Big)^{s}\] in (\ref{zeta-estimate-case3}) comes from the $n=0$ part of the second series in (\ref{zeta-estimate-formula-zero}):
\begin{align*}
&\frac{e^{i\pi\sigma/2}e^{it\log\eta}}{(2\pi)^s}\sum_{j=0}^{N-1}\bigg[\Big(\tfrac{1}{-\frac{it}{z}}\cdot\tfrac{\mathrm{d}}{\mathrm{d}z}\Big)^j\Big(\tfrac{z^{\sigma-1}}{-\frac{it}{z}}\Big)\bigg]_{z=-i\eta} \\
=\,&\frac{e^{i\pi\sigma/2}\eta^{it}}{(2\pi)^s}\sum_{j=0}^{N-1}\bigg[(-it)^{-j-1}\Big(z\cdot\tfrac{\mathrm{d}}{\mathrm{d}z}\Big)^j\big(z^{\sigma}\big)\bigg]_{z=-i\eta} = \frac{e^{i\pi\sigma/2}\eta^{it}}{(2\pi)^s}\sum_{j=0}^{N-1}\bigg[(-it)^{-j-1}\sigma^jz^{\sigma}\bigg]_{z=-i\eta} \\
=\,&\frac{e^{i\pi\sigma/2}\eta^{it}}{(2\pi)^s}\sum_{j=0}^{N-1}\frac{1}{-it}\Big(\frac{\sigma}{-it}\Big)^je^{-i\pi\sigma/2}\eta^\sigma = \frac{\eta^s}{(2\pi)^s}\cdot\frac{1}{-it}\cdot\frac{1-\big(\tfrac{\sigma}{-it}\big)^N}{1-\tfrac{\sigma}{-it}} \\
=\,&-\frac{1}{s}\Big(\frac{\eta}{2\pi}\Big)^s\Big(1-\big(\tfrac{\sigma}{-it}\big)^N\Big),
\end{align*}
and the $t^{-N}$ part is absorbed by the error term.

Thus, we have shown that the results established here are consistent, as expected, with the existing results of \cite{fokas-lenells} for the Riemann zeta function.
\end{remark}

\section*{Acknowledgements}
Both authors gratefully acknowledge the support of the Engineering and Physical Sciences Research Council: the first author via a research student grant, and the second author via a senior fellowship.

%

\begin{thebibliography}{99}

\bibitem{andersson} 
\textsc{J. Andersson}, `Mean value properties of the Hurwitz zeta-function', \textit{Math. Scand.}, 71(2) (1992), pp. 295--300.

\bibitem{balasubramanian} 
\textsc{R. Balasubramanian}, `A note on Hurwitz's zeta-function', \textit{Ann. Acad. Sci. Fenn. Math.}, 4 (1979), pp. 41--44.

\bibitem{davenport} 
\textsc{H. Davenport}, \textit{Multiplicative Number Theory} (Markham, Chicago, 1967).

\bibitem{fokas-lenells} 
\textsc{A. S. Fokas and J. Lenells}, `On the asymptotics to all orders of the Riemann zeta function and of a two-parameter generalization of the Riemann zeta function', \textit{Mem. Amer. Math. Soc.} (submitted).

\bibitem{katsurada} 
\textsc{M. Katsurada and K. Matsumoto}, `Explicit formulas and asymptotic expansions for certain mean square of Hurwitz zeta-functions I', \textit{Math. Scand.}, 78(2) (1996), pp. 161--177.

\bibitem{mezo} 
\textsc{I. Mez\H{o} and A. Dil}, `Hyperharmonic series involving Hurwitz zeta function', \textit{J. Number Theory}, 130(2) (2010), pp. 360--369.

\bibitem{miller} 
\textsc{P. D. Miller}, \textit{Applied Asymptotic Analysis} (AMS, Rhode Island, 2006).

\bibitem{rane} 
\textsc{V. V. Rane}, \textit{Approximate functional equation for the product of functions and divisor problem}, Arxiv preprint, 2005, arXiv:math/0502126 [math.NT], accessed 5 Dec 2016.

\bibitem{siegel} 
\textsc{C. L. Siegel}, `\"{U}ber Riemanns Nachla\ss{ }zur analytischen Zahlentheorie', \textit{Quellen Studien zur Geschichte der Math. Astron. und Phys. Abt. B: Studien 2} (1932), pp. 45--80; reprinted in \textit{Gesammelte Abhandlungen}, Vol. 1., Springer-Verlag, Berlin, 1966.

\bibitem{titchmarsh} 
\textsc{E. C. Titchmarsh}, \textit{The Theory of the Riemann Zeta Function} [2nd ed.] (OUP, New York, 1986).

\bibitem{wang} 
\textsc{Y. Wang}, `On the $2k$-th mean value of Hurwitz zeta function', \textit{Acta Math. Hungar.}, 74(4) (1997), pp. 301--307.

\end{thebibliography}
\end{document}